\documentclass{amsart}

\usepackage[utf8]{inputenc}
\usepackage{amsmath, amssymb, amsthm}
\usepackage[backend=bibtex,firstinits=true]{biblatex}
\usepackage{tikz}
\usepackage{hyperref}

\usetikzlibrary{calc}

\DeclareMathOperator{\Aut}{Aut}
\DeclareMathOperator{\ev}{ev}
\DeclareMathOperator{\Cont}{Cont}

\newcommand{\M}{M}
\newcommand{\C}{C}
\newcommand{\Q}{Q}
\newcommand{\PP}{\mathbb P}
\newcommand{\Mbar}{\overline M}
\newcommand{\Cbar}{\overline C}
\newcommand{\Qbar}{\overline Q}

\newtheorem{lem}{Lemma}
\newtheorem{prop}{Proposition}
\theoremstyle{remark}
\newtheorem{rem}{Remark}
\theoremstyle{definition}
\newtheorem*{defi}{Definition}

\begin{document}

\title{Tautological relations in moduli spaces of weighted pointed curves}
\author{Felix Janda}

\begin{abstract}
  Pandharipande-Pixton have used the geometry of the moduli space of
  stable quotients to produce relations between tautological Chow
  classes on the moduli space $\M_g$ of smooth genus $g$ curves. We
  study a natural extension of their methods to the boundary and more
  generally to Hassett's moduli spaces $\Mbar_{g, \mathbf w}$ of
  stable nodal curves with weighted marked points.

  Algebraic manipulation of these relations brings them into a
  Faber-Zagier type form. We show that they give Pixton's generalized
  FZ relations when all weights are one. As a special case, we give a
  formulation of FZ relations for the $n$-fold product of the
  universal curve over $\M_g$.
\end{abstract}

\maketitle

\section{Introduction}

\subsection{Moduli spaces of curves with weighted markings}

\subsubsection{Definition}

As a GIT variation of the Deligne-Mumford moduli space of stable
marked curves, for any $n$-tuple $\mathbf w = (w_1, \cdots, w_n)$ with
$w_i \in \mathbb Q \cap ]0, 1]$ Hassett \cite{Ha03} has defined a
moduli space $\Mbar_{g, \mathbf w}$, parametrizing nodal semi-stable
curves $C$ of arithmetic genus $g$ with $n$ numbered marked points
$(p_1, \dotsc, p_n)$ in the smooth locus of $C$ satisfying two
stability conditions:
\begin{enumerate}
\item\label{cond:together}The points in a subset $S \subseteq \{1, \dotsc, n\}$ are
  allowed to come together if and only if $\sum_{i \in S} w_i \le 1$.
\item\label{cond:stable} $\omega_C\left(\sum_{i = 1}^n w_ip_i\right)$ is ample.
\end{enumerate}
The second condition implies that the total weight plus the number of
nodes of every genus $0$ component of $C$ must be strictly greater
than 2.

The main cases we have in mind are the usual moduli space $\Mbar_{g,
  n}$ of marked curves, which occurs when all the weights are equal to
1, the case when $\sum_{i = 1}^n w_i \le 1$, which is a
desingularization of the $n$-fold product of the universal curve over
$\Mbar_g$, and the case when $g = 0$, $w_1 = 1$, $w_2 = 1$ and
$\sum_{i = 3}^n w_i \le 1$, which gives the Losev-Manin spaces
\cite{LoMa00}. Moduli spaces mixed pointed curves also naturally
appear when studying moduli spaces of stable quotients
\cite{MOP11}.

We will use various abbreviations for the weight data, like $(\mathbf
w, 1^m)$ for the data with first entries given by $\mathbf w$ and
further $m$ entries of $1$.

\subsubsection{Tautological classes}

Using the universal curve $\pi: \Cbar_{g, \mathbf w} \to \Mbar_{g,
  \mathbf w}$, the $n$ sections $s_i: \Mbar_{g, \mathbf w} \to
\Cbar_{g, \mathbf w}$ corresponding to the markings and the relative
dualizing sheaf $\omega_\pi$ we can define $\psi$- and
$\kappa$-classes:
\begin{align*}
  \psi_i = c_1(s_i^* \omega_\pi) \\
  \kappa_i = \pi_*(c_1(\omega_\pi)^{i + 1})
\end{align*}
Notice that in the case of $\Mbar_{g, n}$ the definition of
$\kappa$-classes is different from the usual definition as in
\cite{ArCo96}.\footnote{The $\kappa$-classes we use here appear
  naturally when pushing forward powers of $\psi$-classes along maps
  forgetting points of small weight whereas the usual $\kappa$-classes
  are convenient when studying push-forwards of powers of
  $\psi$-classes along maps forgetting points of weight equal to
  $1$. The fact that we will mainly consider the first kind of
  push-forwards explains our choice of $\kappa$-classes.}

Each subset $S \subseteq \{1, \dotsc, n\}$ defines a diagonal class
$D_S$ as the class of the locus where all the points of $S$
coincide. By Condition \eqref{cond:together} the class $D_S$ is zero
if and only if $\sum_{i \in S} w_i > 1$.

As $\Mbar_{g, n}$ the moduli space $\Mbar_{g, \mathbf w}$ is
stratified according to the topological type of the curve and each
stratum, indexed by a dual graph $\Gamma$ (see
\cite[Appendix~A]{GrPa03} for a description of dual graphs of strata
of $\Mbar_{g, n}$), is the image of a clutching map
\begin{equation*}
  \xi_\Gamma: \prod_i \Mbar_{g_i, \mathbf w_i} \to \Mbar_{g, \mathbf w}.
\end{equation*}
Here points which are glued together have weight 1. The map
$\xi_\Gamma$ is finite of degree $|\Aut(\Gamma)|$ \footnote{To
  determine the number $|\Aut(\Gamma)|$ of automorphisms of $\Gamma$
  the graph $\Gamma$ should be regarded as a collection of distinct
  half-edges of which some are glued together. For example when $n=0$
  and $\Gamma$ consists of exactly one vertex and one edge, there is
  exactly one non-trivial automorphism, which interchanges the two
  half-edges; accordingly the map $\xi_\Gamma: \Mbar_{g - 1, 2} \to
  \Mbar_g$ is a double cover.}.

\subsection{Formulation of the relations}

To state the Faber-Zagier-type relations on $\Mbar_{g, \mathbf w}$ we
need to introduce several formal power series.

The hypergeometric series $A$ and $B$ already appeared in the original
FZ relations. They are defined by
\begin{equation*}
  A(t) = \sum_{i = 0}^\infty \frac{(6i)!}{(3i)!(2i)!} \left(\frac t{72}\right)^i = 1 + O(t^1), \quad B(t) = \sum_{i = 0}^\infty \frac{(6i)!}{(3i)!(2i)!} \frac{6i + 1}{6i - 1} \left(\frac t{72}\right)^i.
\end{equation*}
We will actually not directly use $B$ in the definition of the
relations but $A$ and a family $C_i$ of series strongly related to $A$
and $B$ which were already used in \cite{PaPi13P} for the proof of the
equivalence between stable quotient and FZ relations on $\M_g$. They
are defined recursively by
\begin{equation*}
  C_1 = C = \frac BA, \quad C_{i + 1} = \left(12 t^2 \frac{\mathrm d}{\mathrm dt} - 4it\right)C_i.
\end{equation*}
Notice that $C_i$ is a multiple of $t^{i - 1}$.

As in \cite{PaPi13P} these series appear in the study of the two
variable functions
\begin{align*}
  \Phi(t, x) =& \sum_{d = 0}^\infty \prod_{i = 1}^d \frac 1{1 - it} \frac{(-1)^d}{d!} \frac{x^d}{t^d} \\
  \gamma =& \sum_{i \ge 1} \frac{B_{2i}}{2i \cdot (2i - 1)} t^{2i - 1} + \log(\Phi),
\end{align*}
where the Bernoulli numbers $B_k$ are defined by
\begin{equation*}
  \frac t{e^t - 1} = \sum_{k = 0}^\infty B_k \frac{t^k}{k!}.
\end{equation*}
These two-variable functions appear both in the localization formula
for stable quotients (see Section~\ref{sec:loc} and \cite{PaPi13P}) and
the $S$-matrix in the equivariant genus $0$ Gromov-Witten theory of
$\mathbb P^1$ (see Section~\ref{sec:locsum}).

Various linear bracket operators are used to insert Chow classes into
the power series in $t$. We define $\{\}_\kappa$, $\{\}_\psi$ and
$\{\}_{D_S}$ for $S \subseteq \{1, \dotsc, n\}$ by
\begin{align*}
  \{t^k\}_\kappa =& \kappa_k t^k \\
  \{t^k\}_\psi =& \psi^k t^k \\
  \{t^k\}_{D_S} =&
  \begin{cases}
    \psi_i^k t^k, & \text{if } S = \{i\}, \\
    (-1)^{|S| - 1} D_S \psi_i^{k - |S| + 1} t^k, & \text{else, where }i \in S,
  \end{cases}
\end{align*}
and linearity. It will moreover be useful to define brackets modified
by a sign $\zeta \in \{\pm 1\}$, denoted by $\{\}_{\kappa}^{\zeta}$,
$\{\}_{\psi}^{\zeta}$ and $\{\}_{D_S}^{\zeta}$ respectively, by
composing the usual bracket operator with the ring map induced by
$t \mapsto \zeta t$. For a power series $F$ in $t$ we will use the
notation $[F]_{t^i}$ for the $t^i$ coefficient of $F$.

\begin{prop}
  \label{prop:fz}
  For any codimension $r$ and the choice of a subset $S \subseteq \{1,
  \dotsc, n\}$ such that $3r \ge g + 1 + |S|$ the class
  \begin{align*}
    \Bigg[\sum_{\substack{\Gamma \\ \zeta: \Gamma \to \{\pm 1\}}}
    \hspace{-2mm}\frac 1{|\Aut(\Gamma)|} \xi_{\Gamma*}\Big(\prod_{v\text{
      vertex}}\hspace{-2mm}\exp(-\{\log(A)\}_{\kappa^{(v)}}^{\zeta(v)}) \sum_{P
    \vdash S_v}\prod_{i \in P} \{C_{|i|}\}_{D_i}^{\zeta(v)}
  \hspace{-1mm}\prod_{e\text{ edge}} \Delta_e\Big)\Bigg]_{t^{r - |E|}}
  \end{align*}
  in $A^r(\Mbar_{g, \mathbf w})$ is zero, where the sum is taken over
  all dual graphs $\Gamma$ of $\Mbar_{g, \mathbf w}$ with vertices
  colored by $\zeta$ with $+1$ or $-1$, the class $\kappa_i^{(v)}$ is
  the $i$-th $\kappa$-class in the factor corresponding to $v$ and
  $S_v$ is $S$ restricted to the markings at $v$. The edge term
  $\Delta_e$ depends only on the $\psi$-classes $\psi_1$, $\psi_2$ and
  colors $\zeta_i = \zeta(v_i) \in \{\pm 1\}$ at the vertices $v_1$,
  $v_2$ joined by $e$ and is defined by
  \begin{equation*}
    2t(\psi_1 + \psi_2) \Delta_e = (\zeta_1 + \zeta_2) \{A^{-1}\}_{\psi_1}^{\zeta_1}\{A^{-1}\}_{\psi_2}^{\zeta_2} + \zeta_1\{C\}_{\psi_1}^{\zeta_1} + \zeta_2\{C\}_{\psi_2}^{\zeta_2}.
  \end{equation*}
\end{prop}

To see that the series $\Delta_e$ is well-defined one can use the
identity $A(t)B(-t) + A(-t)B(t) + 2 = 0$ \cite{Pi12P}. The proof of
Proposition~\ref{prop:fz} gives an alternative more geometric proof.

In the case of $\Mbar_{g, n}$ the relations of
Proposition~\ref{prop:fz} are a reformulation of the part of Pixton's
generalized FZ relations \cite{Pi12P} with empty partition $\sigma$ and
coefficients $a_i$ only valued in $\{0, 1\}$. The set $S$ corresponds
to the set of all $i$ such that $a_i = 1$.

To obtain a set of relations analogous to Pixton's relations we need
to take the closure of the relations of Proposition~\ref{prop:fz}
under multiplication with $\psi$- and $\kappa$-classes and
push-forward under maps forgetting marked points of weight 1. See for
this also the discussion in Section \ref{sec:final} and
\cite[Section~3.5]{PPZ15}.

In total this gives a proof of \cite[Conjecture~1]{Pi12P} in Chow. We
therefore have verified Pixton's remark \cite{Pi12P} that it should be
possible to adapt the stable quotients method to prove that his
generalized relations hold. In cohomology \cite[Conjecture~1]{Pi12P} has
already been established by a completely different method in
\cite{PPZ15}.

\subsection{Plan of Paper}

Section~\ref{sec:sq} introduces stable quotient moduli spaces of
$\mathbb P^1$ with weighted marked points, slightly generalizing the
usual moduli spaces of stable quotients. We define them, sketch their
existence and review structures on them. In Section~\ref{sec:loc} we
review the virtual localization formulas both for stable quotients and
stable maps to $\mathbb P^1$. Section~\ref{sec:open} contains a proof
of Proposition~\ref{prop:fz} restricted to powers of the universal
curve over $\M_g$. The proof in this case is much simpler than the
general case but many parts of it can be referred to later
on. Section~\ref{sec:locsum} provides calculations of localization
sums using Givental's method necessary in the proof of the general
relations. Finally Section~\ref{sec:general} contains the proof of
Proposition~\ref{prop:fz}.

\subsection*{Acknowledgments}

The author would like to thank Rahul Pandharipande for the idea of
considering the universal curve on the moduli space of stable quotient
and recognizing the edge sums appearing in the localization
calculation and Aaron Pixton for explaining many aspects of the
preprint \cite{PaPi11P} and FZ relations in general. The author is
grateful for discussions with Alina Marian, Dragos Oprea and Dimitri
Zvonkine on the moduli of stable quotients.

The author was supported by the Swiss National Science Foundation
grant SNF 200021\_143274.

\section{Moduli spaces of stable quotients}
\label{sec:sq}

\subsection{Introduction}

The proof of Proposition~\ref{prop:fz} is based on the geometry of
stable quotients. These moduli spaces give an alternative
compactification of the space of maps from curves to Grassmannians and
were first introduced in \cite{MOP11}. We will need a combination
of these spaces with Hassett's spaces of weighted stable curves. These
spaces are different from the $\epsilon$-stable quotient spaces
introduced by Toda \cite{To11}, where the stability conditions on
quotient sheaf instead of the points varies. Similar spaces $\Mbar_{g,
  \mathbf w}(\mathbb P^m, d)$ in Gromov-Witten theory have been
studied in \cite{AlGu08}, \cite{BaMa09} and \cite{MuMu08}.

The moduli space $\Qbar_{g, \mathbf w}(\mathbb P^1, d)$ parametrises
nodal curves $C$ of arithmetic genus $g$ with $n$ markings $p_i$
weighted by $\mathbf w$ together with a quotient sequence
\begin{equation*}
  0 \to S \to \mathcal O_C \otimes \mathbb C^2 \to Q \to 0
\end{equation*}
satisfying several conditions:
\begin{itemize}
\item The underlying curve with weighted marked points satisfies all
  the properties of being stable but possibly the ampleness
  condition~\eqref{cond:stable}.
\item $S$ is locally free of rank 1.
\item $Q$ has degree $d$.
\item The torsion of $Q$ is outside the nodes and the markings of weight 1.
\item $\omega_C\left(\sum_{i = 1}^n w_ip_i\right) \otimes S^{\otimes
    (-\varepsilon)}$ is ample for any $0 < \varepsilon \in \mathbb Q$.
\end{itemize}
Isomorphisms of stable quotients are defined by isomorphisms of
weighted stable curves such that the kernels of the quotient maps are
related via pull-back by the isomorphism. In
Section~\ref{sec:sq:const} we give a precise definition for families
and sketch a proof that $\Qbar_{g, \mathbf w}(\mathbb P^1, d)$ is a
proper Deligne-Mumford stack.

\subsection{Structures}

There is a universal curve $\Cbar_{g, \mathbf w}(\mathbb P^1, d)$ over
$\Qbar_{g, \mathbf w}(\mathbb P^1, d)$ with a universal quotient sequence
\begin{equation*}
  0 \to \mathcal S \to \mathcal O_{\Cbar_{g, \mathbf w}(\mathbb P^1, d)} \otimes \mathbb C^2 \to \mathcal Q \to 0.
\end{equation*}
Moreover as in Gromov-Witten theory for each marking $i$ of weight $1$
there is an evaluation map $\ev_i: \Qbar_{g, \mathbf w}(\mathbb P^1,
d) \to \mathbb P^1$, which is defined by noticing that
\begin{equation*}
  0 \to \mathcal S \to \mathcal O_C \otimes \mathbb C^2 \to \mathcal Q \to 0
\end{equation*}
restricted to a marking $p$ of weight $1$ gives a point in $\mathbb
P^1$ since tensoring the above sequence with the residue field $k_p$
at $p$ gives a $k_p$ valued point in the Grassmannian $\mathbb G(1, 2)
= \mathbb P^1$.

There is also a forgetful map $\nu: \Qbar_{g, \mathbf w}(\mathbb P^1,
d) \to \Mbar_{g, \mathbf w}$ forgetting the data of the quotient
sequence and stabilizing unstable components.

Finally, there is a comparison map $c: \Mbar_{g, \mathbf w}(\mathbb
P^1, d) \to \Qbar_{g, \mathbf w}(\mathbb P^1, d)$ from the moduli
space of stable maps of degree $d$ to $\mathbb P^1$. It contracts all
components that would become unstable \footnote{The only components
  which are stable in the Gromov-Witten but not in the stable
  quotients theory are non-contracted components of genus 0 with
  exactly one node and no marking.} and introduces torsion at the
point the component is contracted to according to the degree of the
map restricted to that component.

The open substack $\Q_{g, \mathbf w}(\mathbb P^1, d) \subset \Qbar_{g,
  \mathbf w}(\mathbb P^1, d)$ is defined to be the preimage of the
moduli space of smooth curves $\M_{g, \mathbf w}$ under the forgetful
map $\nu$.

\subsection{Construction}
\label{sec:sq:const}

We want to reduce the existence of moduli spaces of stable quotients
with weighted marked points to that of the usual moduli spaces of
stable quotients. For this we use ideas from \cite{AlGu08} and
\cite{BaMa09}. Because it is not the main topic of this article we
will try to be as brief as possible.

We will allow in this section the weight data also to include $0$. So
with $\mathbf w$ we will denote an $n$-tuple $(w_1, \dotsc, w_n)$ with
$0 \le w_i \le 1$ for all $i \in \{1, \dotsc, n\}$ here. We also
consider more generally stable quotients to $\mathbb P^m$.

\begin{defi}
  An object $(C, s_1, \dotsc, s_n, \mathcal O_C^{m + 1} \to \mathcal
  Q)$ in the stack $\Qbar_{g, \mathbf w}(\mathbb P^m, d)$ over a
  scheme $S$ is a proper, flat morphism $\pi: C \to S$ together with
  $n$ sections $s_i$ and a quotient sequence of quasi-coherent sheaves
  on $C$ flat over $S$
  \begin{equation*}
    0 \to \mathcal S \to \mathcal O_C^{m + 1} \to \mathcal Q \to 0
  \end{equation*}
  such that
  \begin{enumerate}
  \item The fibers of $\pi$ over geometric points are nodal connected
    curves of arithmetic genus $g$.
  \item For any $S \subset \{1, \dotsc, n\}$ such that the
    intersection of $s_i$ for all $i \in S$ is nonempty we must have
    $\sum_{i \in S} w_i \le 1$ and if in addition the intersection
    touches the singular locus of $\pi$ we must have $\sum_{i \in S}
    w_i = 0$.
  \item $\mathcal S$ is locally free of rank $1$.
  \item $\mathcal Q$ is locally free outside the singular locus of
    $\pi$ and of degree $d$.
  \item $\omega_\pi(\sum_{i = 1}^n w_is_i) \otimes \mathcal
    S^{\otimes(-\varepsilon)}$ is $\pi$ relatively ample for any $\varepsilon >
    0$.
  \end{enumerate}
  Two families $(C, s_1, \dotsc, s_n, \mathcal O_C^{m + 1} \to
  \mathcal Q)$, $(C', s'_1, \dotsc, s'_n, \mathcal O_{C'}^{m + 1} \to
  \mathcal Q')$ of stable quotients over $S$ are isomorphic if there
  exists an isomorphism $\phi: C \to C'$ over $S$ mapping $s_i$ to
  $s'_i$ for all $i \in \{1, \dotsc, n\}$ and such that $S$ and
  $\phi^*(S')$ coincide when viewed as subsheaves of $\mathcal
  O_{C'}^{m + 1}$.
\end{defi}

Notice that there is no condition on the sections of weight
$0$. Therefore the space $\Qbar_{g, (\mathbf w, 0^f)}(\mathbb P^m,
d)$ is isomorphic to the $f$-fold power of the universal curve of
$\Qbar_{g, \mathbf w}(\mathbb P^m, d)$ over $\Qbar_{g, \mathbf
  w}(\mathbb P^m, d)$.

In order to prove that $\Qbar_{g, \mathbf w}(\mathbb P^m, d)$ is a
Deligne-Mumford stack it is enough to show that it is locally
isomorphic to a product of universal curves over $\Qbar_{g,
  1^f}(\mathbb P^m, d)$ since this is a Deligne-Mumford stack as shown
in \cite{MOP11} (by realizing it as a stack quotient of a locally
closed subscheme of a relative Quot scheme).
\begin{lem}
  For each point $(C, s_1, \dotsc, s_n, \mathcal O_C^{m + 1} \to
  \mathcal Q)$ in $\Qbar_{g, \mathbf w}(\mathbb P^m, d)$ there exists
  an open neighborhood which is isomorphic to an open substack of
  $\Qbar_{g, \mathbf w'}(\mathbb P^m, d)$ where $\mathbf w'$ is $\{0,
  1\}$ valued.
\end{lem}
\begin{proof}
  The argument is the same as in \cite[Corollary~1.18]{AlGu08}.
\end{proof}

Separatedness follows from the following lemma, which is analogous to
\cite[Proposition~1.3.4]{BaMa09}.
\begin{lem}
  The diagonal $\Delta: \Qbar_{g, \mathbf w}(\mathbb P^m, d) \to
  \Qbar_{g, \mathbf w}(\mathbb P^m, d) \times \Qbar_{g, \mathbf
    w}(\mathbb P^m, d)$ is representable, finite and separated.
\end{lem}
\begin{proof}
  We proceed as in \cite[Proposition~1.3.4]{BaMa09}.

  Let $(C, s_1, \dotsc, s_n, \mathcal O_C^{m + 1} \to \mathcal Q)$ and
  $(C', s'_1, \dotsc, s'_n, \mathcal O_{C'}^{m + 1} \to \mathcal Q')$
  be two stable quotients over a base scheme $S$. We need to show that
  the category
  \begin{equation*}
    \mathrm{Isom}((C, s_1, \dotsc, s_n, \mathcal O_C^{m +
      1} \to \mathcal Q), (C', s'_1, \dotsc, s'_n, \mathcal O_{C'}^{m +
      1} \to \mathcal Q'))
  \end{equation*}
  is represented by scheme finite and separated over $S$.

  The images $T$ and $T'$ of the maps $(\mathcal O_C^{m + 1})^* \to
  S^*$ and $(\mathcal O_{C'}^{m + 1})^* \to (S')^*$ are given by $T =
  S^*(-D)$, $T' = (S')^*(-D')$ for effective divisors $D$, $D'$ of
  some degree $d' \le d$ on $C$ respectively $C'$. From this we can
  construct $d'$ additional sections $s_{n + 1}, \dotsc, s_{n + d'}$
  for $C$ and $s'_{n + 1}, \dotsc, s'_{n + d'}$ for $C'$. As in
  \cite[Proposition~1.3.4]{BeMa96} at least \'etale locally $d -
  d'$ \footnote{Actually $(m + 1)(d - d')$} further sections $s_{n +
    d'}, \dotsc, s_{n + d}$ for $C$ and $s'_{n + d'}, \dotsc, s'_{n +
    d}$ for $C'$ can be constructed by choosing suitable hyperplanes
  $H_i$ in $\mathbb C^{m + 1}$ and marking sections at which the
  quotient is locally free and the quotient sequence coincides with
  the sequence corresponding to the inclusion of $H_i$ in $\mathbb
  C^{m + 1}$.

  By the definition of stable quotients the resulting marked curves
  $(C, s_1, \dotsc, s_{n + d})$ and $(C', s'_1, \dotsc, s'_{n + d})$
  are $(\mathbf w, \varepsilon^d)$-stable. This gives a closed
  embedding
  \begin{align*}
    \mathrm{Isom}((C, s_1, \dotsc, s_n, \mathcal O_C^{m + 1} \to
    \mathcal Q), (C', s'_1, \dotsc, s'_n, \mathcal O_{C'}^{m +
      1} \to \mathcal Q')) \hookrightarrow\\
    \bigsqcup_{\sigma \in S_{d'}} \hspace{-2.5mm}\mathrm{Isom}((C, s_1,
    \dotsc, s_{n + d}), (C', s'_1, \dotsc, s'_n, s'_{n + \sigma(1)},
    \dotsc, s'_{n + \sigma(d')}, s'_{n + d' + 1}, \dotsc, s'_{n + d})).
  \end{align*}
  Since by \cite{Ha03} the right hand side is a scheme finite and
  separated over $S$ so is the left hand side.
\end{proof}

\begin{lem}
  \label{lem:comparison}
  There is a surjective comparison map $c: \Mbar_{g, \mathbf
    w}(\mathbb P^m, d) \to \Qbar_{g, \mathbf w}(\mathbb P^m, d)$.
\end{lem}
\begin{proof}
  This is similar to \cite[Lemma~2.23]{To11}.
\end{proof}

\begin{lem}
  For $\mathbf w' \le \mathbf w$ there is a surjective reduction
  morphism
  \begin{equation*}
    \rho_{\mathbf w'\mathbf w}: \Qbar_{g, \mathbf w}(\mathbb P^m, d) \to \Qbar_{g, \mathbf w'}(\mathbb P^m, d).
  \end{equation*}
\end{lem}
\begin{proof}
  This follows from Lemma~\ref{lem:comparison} and the corresponding
  result for stable maps (see \cite[Proposition~1.2.1]{BaMa09}).
\end{proof}

Since $\Qbar_{g, \mathbf w}(\mathbb P^m, d)$ is proper for $\mathbf w
= 1^n$ the preceding lemma implies that $\Qbar_{g, \mathbf w}(\mathbb
P^m, d)$ is proper in general.

\section{Localization}
\label{sec:loc}

The virtual localization formulas \cite{GrPa99} for $\Qbar_{g,
  \mathbf w}(\mathbb P^m, d)$ and $\Mbar_{g, \mathbf w}(\mathbb P^m,
d)$ are the main tool we use to derive stable quotient relations. The
existence of the necessary virtual fundamental class $[\Mbar_{g,
  \mathbf w}(\mathbb P^m, d)]^{vir}$ for $\Mbar_{g, \mathbf w}(\mathbb
P^m, d)$ has been shown in \cite{AlGu08} and \cite{BaMa09}. For
the existence of a $2$-term obstruction theory of $\Qbar_{g, \mathbf
  w}(\mathbb P^m, d)$ the same arguments as in
\cite[Section~3.2]{MOP11} can be used. They depend on existence of
a $\nu$-relative $2$-term obstruction theory of the $\mathrm{Quot}$
scheme and the non-singularity of the Hassett moduli spaces of
weighted curves. We will now follow Section~7 of \cite{MOP11}.

In this paper we will only look at torus actions on moduli spaces of
stable quotients or stable maps of $\mathbb P^1$ which are induced
from the diagonal action of $\mathbb C^*$ on $\mathbb C^2$ given by
$([x_0: x_1], \lambda) \mapsto [x_0: \lambda x_1]$. By $s$ we will
denote the pull-back of the equivariant class $s \in A_{\mathbb
  C^*}^1(\mathrm{pt})$ defined as the first equivariant Chern class of
the trivial rank 1 bundle on a point space with weight one action of
$\mathbb C^*$ on it.

The equivariant cohomology of $\mathbb P^1$ is generated as an algebra
over $\mathbb Q[s]$ by the equivariant classes $[0]$, $[\infty]$ of
the two fixed points $0$ and $\infty$. These classes satisfy (among
others) the relation $[0] - [\infty] = s$. Localizing by $s$, the
classes $[0]$ and $[\infty]$ give a basis of $A_{\mathbb C^*}^*(\PP^1)$ as
a free $\mathbb Q[s, s^{-1}]$-module.

\subsection{Fixed loci}
\label{sec:loc:fixed}

The fixed loci for the action of $\mathbb C^*$ on $\Mbar_{g, \mathbf
  w}(\mathbb P^1, d)$ and $\Qbar_{g, \mathbf w}(\mathbb P^1, d)$ are
very similar. They are both parametrized by the data of
\begin{enumerate}
\item a graph $\Gamma = (V, E)$,
\item a coloring $\zeta: V \to \{\pm 1\}$,
\item a genus assignment $g: V \to \mathbb Z_{\ge 0}$,
\item a map $d: V \sqcup E \to \mathbb Z_{\ge 0}$,
\item a point assignment $p: \{1, \dotsc, n\} \to V$,
\end{enumerate}
such that $\Gamma$ is connected and contains no self-edges, two
vertices directly connected by an edge do not have the same color
$\zeta$,
\begin{align*}
  g =& h^1(\Gamma) + \sum_{v \in V} g(v), \\
  d =& \sum_{i \in V \sqcup E} d(i), \quad d|_E \ge 1,
\end{align*}
and one further condition which depends on whether we look at
$\Mbar_{g, \mathbf w}(\mathbb P^1, d)$ or $\Qbar_{g, \mathbf
  w}(\mathbb P^1, d)$.

To state the stable quotient condition we need the following
definition: A vertex $v \in V$ is called non-degenerate if the
inequality
\begin{equation*}
  2g(v) - 2 + n(v) + \varepsilon d(v) > 0
\end{equation*}
holds for any $\varepsilon > 0$. Here $n(v)$ is the number of edges at
$v$ plus the number of preimages under $p$ weighted by the
corresponding weight.

Then for the combinatorial data on the stable quotients side we demand
each vertex to be non-degenerate, whereas for the stable maps data the
additional condition is $d|_V = 0$.

In the combinatorial data for $\Mbar_{g, \mathbf w}(\mathbb P^1, d)$
the vertices $v$ of $\Gamma$ correspond to curve components contracted
to the fixed point of $\mathbb P^1$ specified by $\zeta$, that is $0$
for $\zeta(v) = 1$ and $\infty$ for $\zeta(v) = -1$. The edges
correspond to multiple covers of $\mathbb P^1$ ramified only at $0$
and $\infty$ with degree determined by the map $d$. For $\Qbar_{g,
  \mathbf w}(\mathbb P^1, d)$ the vertices of $\Gamma$ correspond to
components $C$ of the curve over which the subsheaf $S$ is an ideal
sheaf of the trivial subsheaf given by one of the two\footnote{These
  correspond to the two fixed points in $\mathbb P^1$.} 1-dimensional
fixed subspaces of $\mathbb C^2$, $\zeta$ specifies which subspace of
$\mathbb C^2$ was chosen and $-d$ the degree of $S$. Edges correspond
to multiple covers of $\mathbb P^1$ ramified in the two torus-fixed
points.

The fixed locus corresponding to the combinatorial data is, up to a
finite map, isomorphic to the product
\begin{equation*}
  \prod_{v \in V} \Mbar_{g(v), (\mathbf w(v), \varepsilon^{d(v)})} / S_{d(v)},
\end{equation*}
where $\mathbf w(v)$ is a $(|p^{-1}(v)| + |\{e\text{ adjacent to
}v\}|)$-tuple, which we will index by $p^{-1}(v) \sqcup \{e\text{
  adjacent to }v\}$, such that $w(v)_i = w_i$ if $i \in p^{-1}(v)$ and
$w(v)_e = 1$ for adjacent edges $e$. The symmetric group $S_{d(v)}$
permutes the $\varepsilon$-stable points. The product should be taken
only over all non-degenerate vertices.

\subsection{The formula}
\label{sec:loc:formula}

The virtual localization formula expresses the virtual fundamental
class as a sum of the (virtual) fundamental classes of the fixed loci
$X$ weighted by the inverse of the equivariant Euler class of the
corresponding virtual normal bundle $N_X$. In order to make sense of
this inverse it is necessary to localize the equivariant cohomology
ring by $s$.

The inverse of the Euler class of $N_X$ is in both cases a product
\begin{equation*}
  \prod_v \Cont(v) \prod_{e\text{ edge}} \Cont(e),
\end{equation*}
for certain vertex and edge contributions depending only on the data
of the graph corresponding to the vertex or edge. The contributions
$\Cont(e)$ and $\Cont(v)$ for $v$ degenerate are pulled back from the
equivariant cohomology of a point. We will not need to know the exact
form of the edge and unstable vertex contributions here apart from the
fact that the contribution of a vertex $v$ with $d(v) = g(v) = 0$ and
$|\mathbf w(v)| = 1$ is equal to $1$.

The non-degenerate vertex contribution is pulled back from
\begin{equation*}
  A_{\mathbb C^*}^*(\Mbar_{g(v), (\mathbf w(v), \varepsilon^{d(v)})} / S_{d(v)}) \otimes \mathbb Q[s, s^{-1}].
\end{equation*}
It is given by
\begin{equation}
  \label{eq:sqvertex}
  \Cont(v) =  (\zeta(v)s)^{g(v) - d(v) - 1} \sum_{j = 0}^\infty \frac{c_i(\mathbb F_{d(v)})}{(\zeta(v)s)^i} \prod_e \frac{\zeta(v)s}{\omega_e^v - \psi_e},
\end{equation}
where $\mathbb F_d$ is the K-theory class $\mathbb F_d = \mathbb E^* -
\mathbb B_d - \mathbb C$, the product is over edges $e$ adjacent to
$v$ and
\begin{equation*}
  \omega_e^v = \frac{\zeta(v)s}{d(e)}.
\end{equation*}
The dual $\mathbb E^*$ of the Hodge bundle $\mathbb E$ has fiber
\begin{equation*}
  (H^0(C, \omega_C))^*
\end{equation*}
over a marked curve $(C, p_i)$, the rank $d$ bundle $\mathbb B_d$ on
$\Mbar_{g, (\mathbf w, \varepsilon^d)} / S_d$ has the fiber
\begin{equation*}
  H^0(C, \mathcal O_C(p_{n + 1} + \dotsb + p_{n + d})|_{p_{n + 1} + \dotsb + p_{n + d}})
\end{equation*}
over a marked curve $(C, p_i)$. The bundle $\mathbb B_d$ can also be
thought as an $S_d$ invariant bundle living on $\Mbar_{g, (\mathbf w,
  \varepsilon^{d})}$. In \cite{MOP11} the Chern classes of this
lifted bundle have been computed:
\begin{equation*}
  c(-\mathbb B_d) = \prod_{i = n + 1}^{n + d} \frac 1{1 - \psi_{n + i} - \sum_{j = n + 1}^{n + i} D_{ij}}.
\end{equation*}
Notice the similarity to the $\Phi$ function. The Chern classes for
$\mathbb B_d$ on $\Mbar_{g, (\mathbf w, \varepsilon^{d})}/S_d$ can be
calculated by push-forward along the finite projection map and
dividing by the degree $d!$.

\subsection{Comparison}

The $\mathbb C^*$ actions on $\Mbar_{g, \mathbf w}(\mathbb P^1, d)$
and $\Qbar_{g, \mathbf w}(\mathbb P^1, d)$ are compatible with the
comparison map $c: \Mbar_{g, \mathbf w}(\mathbb P^1, d) \to \Qbar_{g,
  \mathbf w}(\mathbb P^1, d)$ and $c$ hence maps fixed loci to fixed
loci. The image of a locus corresponding to some combinatorial data is
given by contracting all degenerate vertices, adding the values of $d$
of the contracted vertices and edges to the degree of the image vertex
of the contraction, and replacing the point assignment $p$ by its
composition with the contraction.

The Gromov-Witten stable quotient comparison says that the
contribution of a stable quotient fixed locus to the virtual
fundamental class of $\Qbar_{g, \mathbf w}(\mathbb P^1, d)$ is the sum
of push-forwards of the contributions of all stable maps loci in the
preimage of $c$ to the virtual fundamental class of $\Mbar_{g, \mathbf
  w}(\mathbb P^1, d)$. This in particular implies that
\begin{equation*}
  c_*([\Qbar_{g, \mathbf w}(\mathbb P^1, d)]^{vir}) = [\Mbar_{g, \mathbf w}(\mathbb P^1, d)]^{vir}.
\end{equation*}

\section{The open locus}
\label{sec:open}

We will first restrict ourselves to the special case of $n$-fold
tensor powers $M_{g|n}$ of the universal curve $C_g$ over $M_g$. This
case occurs when the weights are sufficiently small (i.e. $\sum_{i =
  1}^n w_i < 1$) and we restrict ourselves to the locus corresponding
to smooth curves.

\subsection{Statement of the stable quotient relations}

Let us define the bracket operator $\{\}_\Delta$ on $\mathbb Q[\![t,
p_1, \dotsc, p_n]\!]$ in terms of the operators from the introduction by
\begin{align*}
  \{f(t)\}_\Delta =& -\{f(t)\}_\kappa \\
  \left\{f(t) \prod_{i \in S} p_i\right\}_\Delta =& \{f(t)\}_{D_S} \\
  \{p_i^e f(t, p_1, \dotsc, p_n)\}_\Delta =& \{p_i f(t, p_1, \dotsc, p_n)\}_\Delta,
\end{align*}
if $e > 0$. For example $\{t^2p_1^3\}_\Delta = t^2\psi_1^2$ and
$\{t^2p_1p_2\}_\Delta = -D_{12}\psi_1 = -D_{12}\psi_2$.

\begin{prop}
  \label{prop:open:sqrel}
  The relations given by
  \begin{equation*}
    \sum_{\zeta \in \{\pm 1\}} \zeta^{g - 1} \left[\exp\left(-\frac 12 \zeta p + \{\exp(-pD) \gamma(\zeta t, x)\}_\Delta\right)\right]_{t^rx^d\mathbf p^{\mathbf a}} = 0,
  \end{equation*}
  with the differential operator $D = tx\frac{\mathrm d}{\mathrm dx}$
  and $p = p_1 + \dots + p_n$ hold in $\M_{g | n}$ under the condition
  $g - 2d - 1 + |\mathbf a| < r$.
\end{prop}

\subsection{Proof of the stable quotient relations}
\label{sec:open:proof}

We generalize here the localization method of \cite{PaPi13P}.

For each $i \in \{1, \dots, n\}$ we can define a class $s_i \in
A^1(\M_{g | n}(\mathbb P^1, d))$ as the pull-back of $c_1(\mathcal S)$
from the universal curve $\C_{g | n}(\mathbb P^1, d)$ via the $i$-th
section.

For given nonnegative integers $a_i$ let us look at the class
\begin{equation*}
  \mathbf s^{\mathbf a} = \prod_{i = 1}^n (-s_i)^{a_i} \in A^{|\mathbf a|}(\M_{g | n}(\mathbb P^1, d)).
\end{equation*}

The strategy is to study the $\mathbb C^*$ action from
Section~\ref{sec:loc}, to lift $\mathbf s^{\mathbf a}$ to equivariant
cohomology and to write down the localization formula calculating the
push-forward of this class to the equivariant cohomology of $\M_{g |
  n}$. As we have seen the general form of the localization formula is
a sum of contributions from the fixed loci. For each fixed locus a
class from its equivariant Chow ring \emph{localized by the
  localization variable} $s$ pushed forward via the inclusion of the
fixed locus. We get the relations from the fact that the rational
functions in the localization variable we obtain must actually be
polynomials in $s$ after summing over all fixed loci.

\begin{rem}
  In \cite{MOP11} and \cite{PaPi13P} the same strategy was pursued
  but other related classes were chosen to be pushed forward. In
  similar spirit we could add factors of the form $\pi_*(s_{n +
    1}c_1(\omega_\pi)^b)$, where $\pi: \Q_{g | n + 1}(\mathbb P^1, d)
  \to \Q_{g | n}(\mathbb P^1, d)$ is the forgetful map and
  $\omega_\pi$ is the relative dualising sheaf, to the class we are
  pushing forward. However because of the commutative diagram
  \begin{equation*}
    \begin{tikzpicture}[node distance=1.5cm, auto]
      \node (a) {$\Q_{g | n + 1}(\mathbb P^1, d)$};
      \node (b) [below of=a] {$\M_{g | n + 1}$};
      \node (c) [right of=a,node distance=3cm] {$\Q_{g | n}(\mathbb P^1, d)$};
      \node (d) [below of=c] {$\M_{g | n}$};
      \draw[->] (a) to node {$\nu$} (b);
      \draw[->] (a) to node {$\pi$} (c);
      \draw[->] (c) to node {$\nu$} (d);
      \draw[->] (b) to node {$\pi$} (d);
    \end{tikzpicture}
  \end{equation*}
  and the fact that $c_1(\omega_\pi) = \nu^*(\psi_{n + 1})$ these will
  be contained in the completed set of stable quotient
  relations.\footnote{To see this for more than one factor (say $m$
    factors) one needs to first interpret the product as a class on
    the $m$-fold tensor power $X$ of $\Q_{g | n + 1}(\mathbb P^1, d)$
    over $\Q_{g | n}(\mathbb P^1, d)$ and use the birational map from
    $\Q_{g | n + m}(\mathbb P^1, d)$ to $X$.}

  In \cite{PaPi13P} only factors with $a = 1$ were used to derive the
  FZ relations on $\M_g$. The results from
  Section~\ref{sec:open:diffalg} imply that allowing higher values for
  $a$ would not have led to more stable quotient relations.

\end{rem}

The action of $\mathbb C^*$ on $\mathbb P^1$ is induced by the action
of $\mathbb C^*$ on $\mathbb C^2$ given by $([z_0, z_1], \lambda)
\mapsto [z_0, \lambda z_1]$. This naturally induces $\mathbb C^*$
actions not only on $Q_{g | n}(\mathbb P^1, d)$ but also on the
universal curve $\C_{g | n}(\mathbb P^1, d)$ and the universal sheaf
$\mathcal S$. This gives a natural lift of the $s_i$ to equivariant
cohomology and therefore also a natural lift of $\mathbf s^{\mathbf
  a}$. We will not choose this lift but instead
\begin{equation*}
  \tilde {\mathbf s}^{\mathbf a} = \prod_{i = 1}^n \left(-s_i - \frac 12 s\right)^{a_i} \in A_{\mathbb C^*}^{|\mathbf a|}(\M_{g | n}(\mathbb P^1, d)),
\end{equation*}
where the $s_i$ are the natural lifts.

Let us consider the localization formula for this equivariant lift
applied to the push-forward
\begin{equation*}
  \nu_*(\mathbf s^{\mathbf a} \cap [\M_{g | n}(\mathbb P^1, d)]^{vir}) \in A_{2g + 2d - 2 + n - |\mathbf a|}(\M_{g | n}).
\end{equation*}
Let us shorten this by writing $\nu_*^{vir}(\mathbf s^{\mathbf a})$
for this push-forward after capping with the virtual fundamental
class.

Because we assume the strict inequality $\sum_{i = 1}^n w_i < 1$ there
are only two fixed loci with respect to the torus actions in the
description of Section~\ref{sec:loc:fixed}. Concretely in this special
case they correspond to elements $(C, p_1, \dotsc, p_n)$ of $\M_{g |
  n}$ with quotient sequence
\begin{equation*}
  0 \to \mathcal O_C(-p_{n + 1} - \dotsb - p_{n + d}) \to \mathcal O_C^2 \to Q \to 0
\end{equation*}
where the first map factors through $\mathcal O_C$ and the map
$\mathcal O_C \to \mathcal O_C^2$ is induced from one of the two torus
invariant inclusions of $\mathbb C$ as a coordinate axis in $\mathbb
C^2$. Here $\mathcal O_C(-p_{n + 1} - \dotsb - p_{n + d})$ is an ideal
sheaf of $\mathcal O_C$ of degree $d$. Both fixed point loci can be
identified with $\M_{g | n + d}/S_d$ where the symmetric group $S_d$
permutes the last $d$ markings.

Since the graphs corresponding to the fixed point loci have each only
one vertex and no edge we can calculate the inverse of the equivariant
Euler class of the normal bundle to each fixed locus using
\eqref{eq:sqvertex}. It is given by
\begin{equation*}
  (\zeta s)^{g - d - 1} \sum_{j = 0}^\infty \frac{c_i(\mathbb F_d)}{(\zeta s)^i},
\end{equation*}
where $\zeta$ is $+1$ and $-1$ for $0$ and $\infty$ respectively.

Applying the fixed point formula we obtain for the $s^c$ part of the
push-forward 
\begin{equation}
  \label{eq:sqrel1}
  [\nu_*^{vir}(\tilde{\mathbf s}^{\mathbf
    a})]_{s^c} = \frac 1{d!} \sum_\zeta \varepsilon_* \left[\prod_{i = 1}^n \left(-s_it - \frac 12\zeta\right)^{a_i} \sum_{j = 0}^\infty t^j \zeta^{g - d - 1 - j} c_j(\mathbb F_d)\right]_{t^{g - 1 - d - c + |\mathbf a|}},
\end{equation}
where $\varepsilon: \M_{g | n + d} \to \M_{g | n}$ forgets the last
$d$ markings. We have here by abuse of notation denoted similarly
defined classes on $\M_{g | n + d}$ with the same name as on $\M_{g |
  n}(\mathbb P^1, d)$. The expression $(-s_it - \frac 12\zeta)$ comes
from the fact that the torus acts trivially on the subspace of
$\mathbb C^2$ given by $0$ and with weight $1$ on the subspace given
by $\infty$. The equivariant lift of the $s_i$ was chosen in order to
have this symmetric expression.

Since the push-forward must be an honest equivariant class, the classes
\eqref{eq:sqrel1} must be zero if $c < 0$.

Let us package these relations into a power series. We have
\begin{equation}
  \label{eq:sqrel2}
  \prod_{i = 1}^n \frac 1{a_i!} [\nu_*^{vir}(\tilde{\mathbf s}^{\mathbf
    a})]_{s^c} = \sum_\zeta \zeta^{g - d - 1} \varepsilon_* [\exp(-T_1) T_2]_{t^rx^dp^{\mathbf a}},
\end{equation}
with
\begin{align*}
  T_1 =& \sum_{i = 1}^n \left(s_it + \frac 12\zeta\right) p_i,\\
  T_2 =& \sum_{j = 0}^\infty \sum_{d = 0}^\infty (\zeta t)^j c_j(\mathbb F_d) \frac{x^d}{d!}.
\end{align*}
Since
\begin{equation*}
  s_i = \sum_{j = n + 1}^{n + d} D_{ij}
\end{equation*}
we can rewrite $T_1$ as
\begin{equation*}
  T_1 = \sum_{i = 1}^n \sum_{j = n + 1}^{n + d} D_{ij} p_i + \frac 12 \zeta \sum_{i = 1}^n p_i.
\end{equation*}

\subsubsection{Relations between diagonal classes}
\label{sec:open:proof:universal}

In order to better understand $c(\mathbb F_d)$ let us here collect the
universal relations between classes in $A^*(\M_{g | n})$.

The basic relations are
\begin{align*}
  D_{ij} \psi_i = D_{ij} \psi_j = -D_{ij}^2 \\
  D_{ij} D_{ik} = D_{ijk}
\end{align*}
for pairwise different $i$, $j$, $k$ (compare to \cite{Fa99}). Let
$D_{S, a} \in A^a(\M_{g | n})$ by defined for any
$S \subset \{1, \dotsc, n\}$ and $a \ge |S| - 1$ by
\begin{equation*}
  D_{S, a} =
  \begin{cases}
    \psi_i^a, & \text{if } S = \{i\}, \\
    (-1)^{|S| - 1} D_S \psi_i^{a - |S| + 1}, & \text{else, where }i \in \{1, \dotsc, d\}.
  \end{cases}
\end{equation*}
Then $D_{S, a} D_{T, b} = D_{S \cup T, a + b}$ if $S \cap T \neq
\emptyset$.

\begin{lem}
  Each monomial $M$ in diagonal and cotangent line classes in
  $A^*(\M_{g | n})$ can be written as a product
  \begin{equation*}
    M = \pm\prod_{S \in P} D_{S, a(S) + |S| - 1}
  \end{equation*}
  for some partition $P$ of $\{1, \dotsc, d\}$, function $a: P \to
  \mathbb Z_{\ge 0}$ and a suitable choice of sign. This product
  decomposition is unique if we only use the above relations between
  diagonal and cotangent line classes.
\end{lem}
\begin{rem}
  If the partition $P$ is the one element set partition, we say that
  $M$ is connected. We call the factors of the decomposition (or just
  the elements of $P$) the connected components of $M$.
\end{rem}

The push-forward under the forgetful map $\pi: \M_{g | n} \to \M_{g |
  n - 1}$ is given by
\begin{equation*}
  \pi_* D_{S, a} =
  \begin{cases}
    0, & \text{if } n \notin S, \\
    \kappa_{a - 1}, & \text{if } |S \cap \{n\}| = 1, \\
    -D_{S \setminus \{n\}, a - 1}, & \text{else.}
  \end{cases}
\end{equation*}
Here and in the rest of this article $\kappa_{-1}$ is defined to be
zero.

\subsubsection{Simplest relations}
\label{open:proof:simple}

Let us first consider the stable quotient relations in the case of
$\mathbf a = 0$. Then they are simply
\begin{equation*}
  0 = \sum_\zeta \zeta^{g - d - 1} \varepsilon_* \left[\sum_{d = 0}^\infty \sum_{j = 0}^\infty (\zeta t)^j c_j(\mathbb F_d) \frac{x^d}{d!}\right]_{t^rx^d}
\end{equation*}
for
\begin{equation}
  \label{eq:sizecondasimple}
  r > g - 1 - 2d.
\end{equation}

By the definition of $\mathbb F_d$ as a K-theoretic difference of
$\mathbb E^*$, a trivial rank $1$ bundle and $\mathbb B_d$ the inner
sum breaks into two factors. The part corresponding to $\mathbb E^*$
is pulled back from $M_{g | n}$ and does not depend on $d$. Using
Mumford's formula \cite{Mu83} we can therefore rewrite the relations
as: The class
\begin{equation*}
  \sum_\zeta \zeta^{g - d - 1} \left[\exp\Big(-\sum_{i \ge 1} \frac{B_{2i}}{2i(2i - 1)} \kappa_{2i - 1} (t\zeta)^{2i - 1}\Big) \varepsilon_* \sum_{d = 0}^\infty \sum_{j = 0}^\infty (\zeta t)^j c_j(-\mathbb B_d) \frac{x^d}{d!}\right]_{t^rx^d}
\end{equation*}
is zero if \eqref{eq:sizecondasimple} holds.

To deal with the second factor we formally expand it as a power series
in $x$, $t$ and the classes $D_{S, a}$ for various $S \subset \{n + 1,
\dotsc, n + d\}$ and $a \in \mathbb Z_{\ge 0}$. By the exponential
formula and the facts from Section \ref{sec:open:proof:universal} we
have that
\begin{equation*}
  \varepsilon_*\sum_{d = 0}^\infty \sum_{j = 0}^\infty (\zeta t)^j c_j(-\mathbb B_d) \frac{x^d}{d!} = \exp\left(\varepsilon_*\sum_{d = 1}^\infty \sum_{r = 0}^\infty S_d^r D_{\{n + 1, \dotsc, n + d\}, r} (\zeta t)^r \frac{x^d}{d!}\right)
\end{equation*}
where
\begin{equation*}
  \log\left(\sum_{d = 0}^\infty \prod_{i = 1}^d \frac 1{1 - it} \frac{x^d}{d!}\right) = \sum_{d = 1}^\infty \sum_{r = 0}^\infty S_d^r t^r \frac{x^d}{d!}.
\end{equation*}

With $\varepsilon_* D_{\{n + 1, \dotsc, n + d\}, r} = (-1)^{d - 1}
\kappa_{r - d}$ and noticing the similarities between the series defining
$S_d^r$ and $\log(\Phi)$ we obtain
\begin{equation*}
  \varepsilon_* \sum_{d = 0}^\infty \zeta^d \sum_{j = 0}^\infty \frac{(\zeta t)^j}{t^d} c_j(-\mathbb B_d) \frac{x^d}{d!} = \exp\left(-\{\log(\Phi(\zeta t))\}_\kappa\right)
\end{equation*}
and so the stable quotient relations in the case $\mathbf a = 0$ are
\begin{equation*}
  \left[\sum_\zeta \zeta^{g - 1} \exp\left(-\{\gamma(\zeta t)\}_\kappa\right)\right]_{t^rx^d} = 0
\end{equation*}
under Condition~\eqref{eq:sizecondasimple}.

\subsubsection{General relations}

We will investigate how monomials in the $s_i$ affect the push-forward
under $\varepsilon$ of monomials supported only on the last $d$
points.

Notice that for each partition $P$ of $\{n + 1, \dotsc, n + d\}$ we have
\begin{equation}
  \label{eq:split}
  \exp\left(-\sum_{i = 1}^n s_i\right) = \prod_{S \in P} \exp\left(-\sum_{i = 1}^n \sum_{j \in S} D_{ij}\right),
\end{equation}
and each factor is pulled back via the map forgetting all points in
$\{n + 1, \dotsc, n + d\}$ not in $S$.

Moreover notice that if $M$ is a connected monomial supported in the last $d$
marked points with $\varepsilon_* M = [-\{f(t)\}_\kappa]_{t^r} =
[\{f(t)\}_{\Delta}]_{t^r}$, we have
\begin{equation}
  \label{eq:conn}
  [\varepsilon_*((-s_ip_i)^e x^d M)]_{x^d} = [\varepsilon_*((-dD_{i, n + 1}p_i)^e x^d M)]_{x^d} = [\{(p_iD)^e(x^df(t))\}_{\Delta}]_{t^rx^d}.
\end{equation}

From \eqref{eq:split}, \eqref{eq:conn} and the identity
\begin{equation*}
  \exp(pD)\exp(f(t, x)) = \exp(\exp(pD)f(t, x))
\end{equation*}
we obtain the general stable quotient relations
\begin{equation*}
  \sum_\zeta \zeta^{g - 1} \left[\exp\left(-\frac 12 \zeta p + \{\exp(pD) \gamma(\zeta t, x)\}_\Delta\right)\right]_{t^rx^d\mathbf p^{\mathbf a}} = 0,
\end{equation*}
if
\begin{equation}
  \label{eq:sizeconda}
  r > g - 1 - 2d + |\mathbf a|.
\end{equation}

\subsection{Evaluation of the relations}

\subsubsection{Minor simplification}

Notice that in the relations of Proposition~\ref{prop:open:sqrel} the
summands in the $\zeta$ sum are equal up to a sign
\begin{equation*}
  (-1)^{g - 1 + r + |\mathbf a|}.
\end{equation*}
Therefore the relations are in fact zero if $g + r + |\mathbf a|
\equiv 0 \pmod{2}$, and in the case
\begin{equation}
  \label{eq:modcond}
  g - 1 + r + |\mathbf a| \equiv 0 \pmod{2}
\end{equation}
we can reformulate them to
\begin{equation*}
  \left[\exp\left(-\frac 12 p + \{\exp(pD) \gamma(t, x)\}_\Delta\right)\right]_{t^rx^d\mathbf p^{\mathbf a}} = 0.
\end{equation*}

\subsubsection{Differential algebra}
\label{sec:open:diffalg}

In this section we will establish that it is enough to consider the
stable quotient relations in the case that $a_i < 2$ for all
$i \in \{1, \dotsc, m\}$.

The series $\delta = D\gamma - \frac 12$ satisfies the differential
equation
\begin{equation*}
  D \delta = -\delta^2 + x + \frac 14,
\end{equation*}
as can be derived from the differential equation
\begin{equation*}
  D(\Phi - D\Phi) = -x\Phi,
\end{equation*}
which is satisfied by $\Phi$ as seen by looking at its series
definition.

Reformulating the relations with $\delta$ gives
\begin{equation*}
  \left[\exp\left(-\{\gamma\}_\kappa + \sum_{i = 1}^\infty \frac 1{i!}\{p^iD^{i - 1} \delta\}_{\Delta}\right)\right]_{t^rx^d\mathbf p^{\mathbf a}} = 0
\end{equation*}
if \eqref{eq:sizeconda} and \eqref{eq:modcond} hold.

Let us consider
\begin{equation*}
  G = \frac{\frac{\partial^2}{\partial p_j^2} F}F,
\end{equation*}
with
\begin{equation*}
  F = \exp\left(-\{\gamma\}_\kappa + \sum_{i = 1}^\infty \frac 1{i!}\{p^iD^{i - 1} \delta\}_{\Delta}\right).
\end{equation*}
We have that
\begin{multline*}
  G = \frac{\partial^2}{\partial p_j^2} \sum_{i = 1}^\infty \frac 1{i!}\{p^iD^{i - 1} \delta\}_{\Delta} + \left(\frac{\partial}{\partial p_j} \sum_{i = 1}^\infty \frac 1{i!}\{p^iD^{i - 1} \delta\}_{\Delta}\right)^2 \\
  = \{\exp(pD) D\delta\}_{\Delta_j} + \{\exp(pD)\delta\}_{\Delta_j}^2,
\end{multline*}
where the operator $\{\}_{\Delta_j}$ is defined by $\{f\}_{\Delta_j} =
p_j^{-1}\{p_jf\}_\Delta$. The bracket $\{\}_{\Delta_j}$ and squaring
commute because, informally said, $\{\}_{\Delta_j}$ connects any term
to $j$. So we obtain
\begin{align*}
  G =& \{\exp(pD)D \delta + \left(\exp(pD) \delta\right)^2\}_{\Delta_j} = \{\exp(pD)(D \delta + \delta^2)\}_{\Delta_j}\\
  =& \left\{\exp(pD)\left(x + \frac 14\right)\right\}_{\Delta_j} = \frac 14 + x\{\exp(pt)\}_{\Delta_j},
\end{align*}
Therefore we have that
\begin{equation*}
  \frac{\partial^2}{\partial p_j^2} F = \left(\frac 14 + x\{\exp(pt)\}_{\Delta_j}\right)F.
\end{equation*}
and can express the relations for $a_j \ge 2$ in terms of lower
relations multiplied by monomials in cotangent line and diagonal
classes.

\subsubsection{Substitution}
\label{sec:open:subst}

The differential equation satisfied by $-t\gamma$ has been studied by
Ionel in \cite{Io05}. In \cite{PaPi13P} her results were extended
to give formulas for $D^i \gamma$. We will collect some of their
results on $\gamma$ and its derivatives here.

With the new variables
\begin{equation*}
  u = \frac t{\sqrt{1 + 4x}}, \qquad y = \frac{-x}{1 + 4x}
\end{equation*}
one can write
\begin{equation*}
  \gamma = \frac 1t (t\gamma)(0, x) + \frac 14 \log(1 + 4y) + \sum_{k = 1}^\infty \sum_{j = 0}^k c_{k, j} u^k y^j,
\end{equation*}
for some coefficients $c_{k, j}$, which are defined to vanish outside
the summation region used above. Furthermore, we have for the
derivatives of $\delta$
\begin{equation*}
  D^{i - 1} \delta = (1 + 4y)^{-\frac i2}\left(\sum_{j = 0}^{i - 1} b_j^i u^{i - 1}y^j - \sum_{k = 0}^\infty \sum_{j = 0}^{k + i} c_{k, j}^i u^{k + i}y^j\right) =: (1 + 4y)^{-\frac i2}\delta_i,
\end{equation*}
for some coefficients $b_j^i$, $c_{k, j}^i$.

We will also need a result by Ionel relating coefficients of a power
series $F$ in $x$ and $t$ before and after the variable substitution:
\begin{equation*}
  [F]_{t^rx^d} = (-1)^d [(1 + 4y)^{\frac{r + 2d - 2}2} F]_{u^ry^d}
\end{equation*}

\smallskip

Let us now apply these formulas to the relations. Using the fact that
$\kappa_{-1} = 0$ we get
\begin{equation*}
  \left[(1 + 4y)^e\exp\left(-\left\{\sum_{k = 1}^\infty \sum_{j = 0}^k c_{k, j} u^k y^j\right\}_\kappa + \sum_{i = 1}^\infty \frac 1{i!}\{p^i\delta_i\}_{\Delta}\right)\right]_{u^ry^d\mathbf p^{\mathbf a}} = 0,
\end{equation*}
under conditions \eqref{eq:sizeconda} and \eqref{eq:modcond}, where
the exponent $e$ is
\begin{equation*}
  e = \frac{r + 2d - 2}2 - \frac{\kappa_0}4 - \frac{|\mathbf a|}2 = \frac{r + 2d - 1 - g - |\mathbf a|}2.
\end{equation*}

\subsubsection{Extremal coefficients}
\label{sec:open:ext}

It is noticeable that in the series
\begin{equation*}
  \sum_{k = 1}^\infty \sum_{j = 0}^k c_{k, j} u^k y^j, \quad \sum_{j = 0}^{i - 1} b_j^i u^{i - 1}y^j - \sum_{k = 0}^\infty \sum_{j = 0}^{k + i} c_{k, j}^i u^{k + i}y^j
\end{equation*}
appearing in the formulas for $\gamma$ and $D^{i -1}\delta$ the
$y$-degree is bounded from above by the $u$-degree. We will only be
interested in the extremal coefficients. Here the $A$ and $C_i$ come
into the picture.

In \cite{Io05} it has been proven that
\begin{equation*}
  \log(A(t)) = \sum_{k = 1}^\infty c_{k, k} t^k
\end{equation*}
and in \cite{PaPi13P} it is shown that
\begin{equation*}
  2^{-i} C_i(t) = b_{i - 1}^i t^{i - 1} - \sum_{k = 0}^\infty c_{k, k + i}^i t^{k + i}.
\end{equation*}

We see that in the exponential factor of the relations for each
summand the $y$-degree is bounded from above by the
$u$-degree. Furthermore the exponent $e$ of the $(1 + 4y)$-factor is
integral and positive by \eqref{eq:modcond} and
\eqref{eq:sizeconda}. This implies that the relation is zero unless
\begin{equation}
  \label{eq:sizecondb}
  3r \ge g + 1 + |\mathbf a|
\end{equation}
holds. With the following lemma we can extract the extremal part of
the relations.

\begin{lem}
  Fix any $\mathbb Q$ algebra $A$, any $F \in A[y]$ and any $c
  \in \mathbb Z_{\ge 0}$. The relations
  \begin{equation*}
    [(1 + 4y)^d F]_{y^d} = 0
  \end{equation*}
  for all $d > c$ imply $F = 0$.
\end{lem}
\begin{proof}
  The relations can be rewritten to
  \begin{equation*}
    \left[\left(\frac 1y + 4\right)^d F\right]_{y^0} = 0.
  \end{equation*}
  The lemma follows by the fact that $F$ is a polynomial in $y$, and
  linear algebra.
\end{proof}

Using the lemma and the formulas for the extremal coefficients we
obtain that the relations
\begin{equation*}
  \left[\exp(-\{\log(A)\}_\kappa) \exp\left(\sum_{i = 1}^\infty \frac 1{i!}\{p^iC_i\}_{\Delta}\right)\right]_{z^r\mathbf p^{\mathbf a}} = 0
\end{equation*}
hold under the conditions \eqref{eq:modcond} and \eqref{eq:sizecondb}.

Ignoring terms of higher order in the $p_i$ the second factor can be
rewritten
\begin{align*}
  \exp\left(\sum_{i = 1}^\infty \frac 1{i!}\{p^iC_i\}_{\Delta}\right) \equiv \exp\left(\sum_{\emptyset \neq S \subseteq \{1, \dotsc, n\}} \{p^S C_{|S|}\}_\Delta\right) \\
  \equiv \sum_{S \subseteq \{1, \dotsc, n\}} p^S \sum_{P \vdash S}
  \prod_{i \in P} \{C_{|i|}\}_{D_i}.
\end{align*}

Thus we finally obtain that the FZ-type relations
\begin{equation*}
  \left[\exp(-\{\log(A)\}_\kappa) \sum_{P \vdash S} \prod_{i \in P} \{C_{|i|}\}_{D_i}\right]_{z^r} = 0,
\end{equation*}
hold for any $S \subseteq \{1, \dotsc, n\}$ if \eqref{eq:modcond} and
\eqref{eq:sizecondb}.

\section{Localization sums}
\label{sec:locsum}

In the derivation of the more general stable quotient relations we
will need to deal with two types of localization sums related to the
nodes and marked points, respectively. To keep the length of the proof
of the stable quotient relations more reasonable we will deal with
them here. The sums are genus independent and have been studied more
generally by Givental \cite{Gi98}. We apply here his method in a
special case. See also \cite{Gi01a} and \cite{LePa04P}.

Let $N_{\mathbb C^*}^*(\mathbb P^1)$ be the Novikov ring of $\mathbb
P^1$ with values in $\mathbb C[s, s^{-1}]$. We define a formal
Frobenius manifold structure on $X = A_{\mathbb C^*}^*(\mathbb P^1)
\otimes N_{\mathbb C^*}^*(\mathbb P^1)$ over $N_{\mathbb
  C^*}^*(\mathbb P^1)$ using equivariant Gromov-Witten theory. The
$N_{\mathbb C^*}^*(\mathbb P^1)$-module $X$ is free of dimension two
with basis $\{\phi_0, \phi_\infty\}$ for $\phi_i = [i]/e_i$, where
$e_i$ is the equivariant Euler class of the tangent space of $\mathbb
P^1$ at $i$. We denote the corresponding coordinate functions by
$y_0$, $y_\infty$. This gives a basis $\{\frac{\partial}{\partial
  y_0}, \frac{\partial}{\partial y_\infty}\}$ of the space of vector
fields on $X$. The metric is given in this basis by
\begin{equation*}
  g =
  \begin{pmatrix}
    s^{-1} & 0 \\
    0 & -s^{-1}
  \end{pmatrix}
  =
  \begin{pmatrix}
    e_0^{-1} & 0 \\
    0 & e_\infty^{-1}
  \end{pmatrix}.
\end{equation*}

The primary equivariant Gromov-Witten potential is
\begin{equation*}
  \sum_{n = 0}^\infty \frac 1{n!} \sum_{d = 0}^\infty x^d \int\limits_{[\Mbar_{0, n}(\mathbb P^1, d)]^{vir}} \prod_{i = 1}^n \ev_i^*(y_0\phi_0 + y_\infty\phi_\infty) = \frac{y_0^3}{6s} - \frac{y_\infty^3}{6s} + xe^w,
\end{equation*}
where we have set $w = (y_0 - y_\infty)/s$ (compare to
\cite{OkPa06b}). Thus we have
\begin{equation*}
  \alpha =
  \begin{pmatrix}
    \mathrm dy_0 + \frac x{e_0} e^w \mathrm dw & \frac x{e_\infty} e^w \mathrm dw \\
    \frac x{e_0} e^w \mathrm dw & \mathrm dy_\infty + \frac x{e_\infty} e^w \mathrm dw
  \end{pmatrix},
\end{equation*}
for the matrix of one forms
\begin{equation*}
  \alpha = \sum_i A_i \mathrm dy_i,
\end{equation*}
where the matrices $A_i$ are defined by the quantum product
\begin{equation*}
  \frac{\partial}{\partial y_i} \star \frac{\partial}{\partial y_a} = \sum_b [A_i]_a^b \frac{\partial}{\partial y_b}.
\end{equation*}

Using $\alpha$ we can compactly write down the differential equation
for the $S$-matrix
\begin{equation*}
  (t\mathrm d  - \alpha)S = 0,
\end{equation*}
with initial condition $S(0, 0) = \mathrm{Id}$.

If we set
\begin{equation*}
  S = \begin{pmatrix} S_0^0 & S_\infty^0 \\ S_0^\infty & S_\infty^\infty \end{pmatrix},
\end{equation*}
this gives the system
\begin{equation*}
  t\mathrm dS_i^j = S_i^j \mathrm dy_j + \sum_k S_i^k \frac x{e_k} e^w \mathrm dw,
\end{equation*}
with unique solution
\begin{equation*}
  S_i^j = e^{\frac{y_i}t}\left(\left(\frac{1 + \zeta_i\zeta_j}2 - tx\frac{\mathrm d}{\mathrm dx}\right) \Phi\right)\left(-\frac t{e_i}, \frac{x e^w}{e_i^2}\right)
\end{equation*}
under the initial conditions. Here the signs $\zeta_i \in \{\pm 1\}$
are defined by $\zeta_i = \frac{e_i}s \in \{\pm 1\}$.

There is a set of canonical coordinates $u^i$ on $X$ defined by the
localization sums
\begin{equation*}
  u^i = \sum_{n = 0}^\infty \sum_{d = 0}^\infty x^d \sum_{\Gamma \in G_{0, n + 2}^{u^i}(\mathbb P^1, d)} \Cont_\Gamma \left(e_i \frac{\partial^2 F_0}{(\partial y_i)^2}\right),
\end{equation*}
where $G_{0, n + 2}^{u^i}(\mathbb P^1, d)$ is the set of $(n +
2)$-point (each of weight 1) degree $d$ localization graphs with the
first two points on a single component contracted to $i$ and
$\Cont_\Gamma$ stands for the contribution of a graph $\Gamma$ in the
localization formula. By the Gromov-Witten stable quotient comparison
and since there is only one fixed locus on the stable quotient side we
have
\begin{equation*}
  u^i(0, 0) = e_i e_i^{- d - 1} \sum_{d = 1}^\infty \frac{x^d}{d!} \int\limits_{\Mbar_{0, 2 | d}} \sum_{j = 0}^\infty \frac{c_j(\mathbb F_d)}{e_i^j},
\end{equation*}
where $\Mbar_{0, 2 | d} = \Mbar_{0, (1^2, \varepsilon^d)}$ is a
Losev-Manin space. Since the Hodge bundle is trivial on $\Mbar_{0, 2
  | d}$ this simplifies to
\begin{equation*}
  u^i(0, 0) = e_i e_i^{- d - 1} \sum_{d = 1}^\infty \frac{x^d}{d!} \int\limits_{\Mbar_{0, 2 | d}} \sum_{j = 0}^\infty \frac{c_j(-\mathbb B_d)}{e_i^j}.
\end{equation*}
Recalling the definition of $\mathbb B_d$ we can write the integrand
on the right hand side as a sum of monomials in $\psi$- and diagonal
classes. The integral of such a monomial vanishes unless it is the
diagonal where all $d$ points come together. The constants $-C_d^{-1}$
which are defined from $\log(\Phi)$ by
\begin{equation*}
  \log(\Phi(t, x)) = \sum_{d = 1}^\infty \sum_{r = -1}^\infty C_d^r t^r \frac{x^d}{d!}
\end{equation*}
exactly count these contributions. So we get
\begin{equation}
  \label{eq:canonicalcoord}
  u^i(0, 0) = -\sum_{d = 1}^\infty C_d^{-1} \frac{x^d}{d!} e_i^{-2d + 1}.
\end{equation}

Using the $S$-matrix and the $u^i$ we can now calculate the series we
are interested in. We will not regard $S$ and the $u^i$ outside $(0,
0)$, so let us write from now on $S_i^j$ for $S_i^j(0, 0)$ and $u^i$
for $u^i(0, 0)$.

The first series $P^{ij}$ will be needed to deal with stable quotient
localization chains containing one of the weight $1$ marked points. We
have
\begin{equation*}
  P^{ij}(t, x) := \frac{1 + \zeta_i\zeta_j}2 + \sum_{d = 1}^\infty x^d \hspace{-4mm}\sum_{\Gamma \in G_{0, 2}^{P^{ij}}(\mathbb P^1, d)} \frac{e_i e^{u^i/\omega_\Gamma}}{\omega_\Gamma - t} \Cont_\Gamma\left(\frac{\partial F_0}{\partial y_i\partial y_j}\right) = e^{u^i/t} \zeta_i \zeta_j S_i^j(-t),
\end{equation*}
where $G_{0, 2}^{P^{ij}}(\mathbb P^1, d)$ is the set of all
localization graphs with the first marking on a valence 2 vertex at
$i$ and the second marking at $j$ and $\omega_\Gamma$ is the $\omega$
as in \eqref{eq:sqvertex} corresponding to the flag at the first
marking.

The second series is needed to deal with chains at the nodes of the
curve. We have
\begin{align*}
  E^{ij}(t_1, t_2, x) := \sum_{d = 1}^\infty x^d \sum_{\Gamma \in G_{0, 2}^{E^{ij}}(\mathbb P^1, d)} \frac{e_ie^{u^i/\omega_{\Gamma, 1}}}{\omega_{\Gamma, 1} - t_1} \frac{e_je^{u^j/\omega_{\Gamma, 2}}}{\omega_{\Gamma, 2} - t_2} \Cont_\Gamma\left(\frac{\partial F_0}{\partial y_i\partial y_j}\right) \\
  = \frac s{t_1 + t_2} \left(\frac{\zeta_i + \zeta_j}2 - e^{u^i/t_1 + u^j/t_2} \zeta_i\zeta_j\sum_{\mu} S_i^\mu(-t_1) \zeta_\mu S_j^\mu(-t_2)\right),
\end{align*}
where $G_{0, 2}^{E^{ij}}(\mathbb P^1, d)$ is the set of all
localization graphs with both markings on valence 2 vertices, the
first marking mapped to $i$, the second marking to $j$, and
$\omega_{\Gamma, 1}$, $\omega_{\Gamma, 2}$ are $\omega$ as in
Section~\ref{sec:loc:formula} corresponding to the flag at the first
and second marking respectively.

Let us now simplify the expressions for $P^{ij}(t, x)$ and
$E^{ij}(t_1, t_2, x)$ using the explicit $S$-matrix and canonical
coordinates at 0. For $P^{ij}(t, x)$ we obtain
\begin{align*}
  P^{ij}(t, x) =& \exp\left(-\sum_{d = 1}^\infty C_d^{-1} t^{-1} \frac{x^d}{d!} e_i^{-2d}\right) \left(\left(\frac{1 + \zeta_i\zeta_j}2 - t \zeta_i\zeta_j x\frac{\mathrm d}{\mathrm dx}\right) \Phi\right)\left(\frac t{e_i}, \frac x{e_i^2}\right) \\
  =& \left(\left(\frac 12 - \zeta_i\zeta_j \delta\right) \Phi'\right)\left(\frac t{e_i}, \frac x{e_i^2}\right),
\end{align*}
where $\Phi'$ is defined by
\begin{equation*}
  \Phi'(t, x) = \exp\left(-\sum_{d = 1}^\infty C_d^{-1} t^{-1} \frac{x^d}{d!}\right) \Phi(t, x) = \exp\left(\sum_{d = 1}^\infty \sum_{r = 0}^\infty C_d^r t^r \frac{x^d}{d!}\right).
\end{equation*}
Similarly we obtain
\begin{align*}
  \frac{t_1 + t_2}s E^{ij}(t_1, t_2, x) = &\frac{\zeta_i + \zeta_j}2 + \\
  &\Phi'\left(\frac{t_1}{e_i}, \frac x{e_i^2}\right)\Phi'\left(\frac{t_2}{e_j}, \frac x{e_j^2}\right) \left(\zeta_i\delta\left(\frac{t_1}{e_i}, \frac x{e_i^2}\right) + \zeta_j\delta\left(\frac{t_2}{e_j}, \frac x{e_j^2}\right)\right).
\end{align*}

\section{The general case}
\label{sec:general}

We now extend the relations of Section~\ref{sec:open} to the boundary
and allow nearly arbitrary weights. More precisely we will assume that:
\begin{equation}
  \label{eq:weightcond}
  \text{If for some $S \subset \{1, \dotsc, n\}$ we have $\sum_{i \in S} w_i =
1$, we must have $|S| = 1$.}
\end{equation}
We still obtain relations for any weight data since one can always
modify $\mathbf w$ a bit such that the moduli space $\Mbar_{g, \mathbf
  w}$ does not change (whereas $\Qbar_{g, \mathbf w}(\mathbb P^1, d)$
will change). It is even not necessary to allow that there exist
points of weight 1 at all but it is interesting to see different ways
of obtaining the same relations.

\subsection{Statement of the stable quotient relations}
\label{sec:gen:statement}

Let $\mathcal G$ be the set of stable graphs describing the strata on
$\Mbar_{g, \mathbf w}$. The data of $\mathcal G$ in particular
includes a map $p$ from $\{1, \dotsc, n\}$ to the set of vertices as
in Section~\ref{sec:loc:fixed}. Let us assume that the first $n'$
points are of weight different from $1$ and the remaining $n''$ points
are of weight $1$.

\begin{prop}
  \label{prop:gen:sqrel}
  For a codimension $r$, a degree $d$, an $n$-tuple $\mathbf a$ such
  that $g - 2d - 1 + |\mathbf a| < r - |E|$, in $A^r(\Mbar_{g, \mathbf
    w})$ the relation
  \begin{multline*}
    0 = \Bigg[\sum_{\substack {\Gamma \in \mathcal G \\ \zeta: \Gamma \to
        \{\pm 1\}}} \frac 1{\Aut(\Gamma)} \xi_{\Gamma *}\Big(\prod_v
    \mathrm{Vertex^3}_v^{\zeta(v)} \prod_e \mathrm{Edge^3}_e^{\zeta(v_1),
      \zeta(v_2)} \\
    \prod_{i = n' + 1}^n \mathrm{Point^3}(\psi_i\zeta(p(i))t, p_i)
    \Big)\Bigg]_{t^{r - |E|}x^d\mathbf p^{\mathbf a}}
  \end{multline*}
  holds, where $\mathrm{Vertex^3}_v^\zeta$ is a product
  \begin{equation*}
    \mathrm{Vertex^3}_v^\zeta = \zeta^{g(v) - 1} \exp\left(-\frac 12 \zeta p_{(v)} + \{\exp(p_{(v)}D) \gamma(t\zeta, x)\}_{\Delta^{(v)}} + V^\zeta\right)
  \end{equation*}
  in $A^*(\Mbar_{g(v), \mathbf w_v})[x, t, \mathbf p]$, where
  $p_{(v)} = \sum_{n' \ge i \in p^{-1}(v)} p_i$, $\{\}_{\Delta^{(v)}}$
  is defined identically compared to $\{\}_\Delta$ but $\kappa_j$ is
  replaced by $\kappa_j^{(v)}$ --- the $\kappa_j$ class at $v$ --- and
  \begin{equation*}
    V^\zeta = - \sum_{i \ge 1} \frac{B_{2i}}{2i \cdot (2i - 1)}
    \sum_{\Delta \in \mathcal{DG}}
    \frac 1{\Aut(\Delta)}\xi_{\Delta, *}\left(\frac{\psi_a^{2i - 1} + \psi_b^{2i -
          1}}{\psi_a + \psi_b}\right) (\zeta t)^{2i - 1}.
  \end{equation*}
  Here $\mathcal{DG} \subseteq \mathcal G$ is the set of graphs
  corresponding to divisor classes, i.e. graphs with only one edge,
  and $\psi_a$, $\psi_b$ are the two cotangent line classes
  corresponding to the unique node corresponding to a divisor. The
  edge and point series are given by
  \begin{align*}
    t(\psi_1 + \psi_2) \mathrm{Edge^3}_e^{\zeta_1, \zeta_2} =
    (\Phi'(t\zeta_1\psi_1)\Phi'(t\zeta_2\psi_2))^{-1}\frac{\zeta_1 + \zeta_2}2 + \\
    \zeta_1 \delta(t\zeta_1\psi_1, x) + \zeta_2 \delta(t\zeta_2\psi_2, x),
  \end{align*}
  where $\Phi'$ is from Section~\ref{sec:locsum}, and
  \begin{equation*}
    \mathrm{Point^3}(t, p) = 1 + p \delta(t, x).
  \end{equation*}
\end{prop}

\subsection{Proof of the stable quotient relations}

In Section~\ref{sec:open:proof} we looked at the integrand $\mathbf
s^{\mathbf b}$ coming from powers of the pulled back first Chern class
of the universal sheaf $\mathcal S$ over $\C_{g | n}(\mathbb P^1,
d)$. This works also well for $\Qbar_{g | \mathbf w}(\mathbb P^1, d)$
if no point is of weight $1$. For the points of weight $1$ we however
need to choose different classes since by the stability conditions the
torsion of $\mathcal S$ must be away from the sections of points of
weight $1$. Instead we will pull back classes from $\mathbb P^1$ via
evaluation maps.

For an $n$-tuple $\mathbf a$, which can be split into an $n'$-tuple
$\mathbf a'$ and an $n''$-tuple $\mathbf a''$, we will thus study the
equivariant class
\begin{equation*}
  \varphi(\mathbf a) := \tilde{\mathbf s}^{\mathbf a'} \prod_{i = n' + 1}^n \mathrm{ev}_i^*\left(\left(\frac{[0] + [\infty]}2\right)^{a_i}\right) \in A_{\mathbb C^*}^{|\mathbf a|}(\Qbar_{g, \mathbf w}(\mathbb P^1, d)),
\end{equation*}
where $[0]$ and $[\infty]$ are the equivariant classes of $0$ and
$\infty$ respectively, and $\tilde{\mathbf s}^{\mathbf a'}$ is the
equivariant class from Section~\ref{sec:open:proof}. Since $A_{\mathbb
  C^*}^*(\mathbb P^1) \otimes \mathbb Q[s, s^{-1}]$ is a two
dimensional $\mathbb Q[s, s^{-1}]$ module we do not lose any relations
by considering only the case when $\mathbf a''$ is $\{0, 1\}$ valued.

As before we consider the $s^c$ part of the push-forward of
$\varphi(\mathbf a)$ for $c < 0$ when using the virtual
localization formula.

In order to gain overview over the in $d$ monotonously growing number
of fixed loci we sort them according to the stratum of $\Mbar_{g,
  \mathbf w}$ they push forward to. Let us consider stable graphs
$\Gamma = (V, E)$ of $\Mbar_{g, \mathbf w}$ together with a coloring
$\zeta: V \to \{\pm 1\}$ and a degree assignment $d: V \sqcup E \sqcup
\{n' + 1, \dots, n\} \to \mathbb Z_{\ge 0}$. This data records
coloring and the $d|_V$ from a stable quotient graph, the total
degrees of the chains which destabilize to a node or a weight $1$
marked point when pushing forward to $\Mbar_{g, \mathbf w}$. To get
back to a stable quotient graph one needs to choose for each edge and
each weight $1$ marked point of degree $d$ a splitting $d = d_{e_1} +
d_{v_1} + \dotsb + d_{v_{\ell - 1}} + d_{e_\ell}$ of $d$ corresponding
to the degrees on the chain (see Figure~\ref{fig:chain}). Because of the
conditions on the coloring of a stable graph there is a mod 2
condition on the length $\ell$ of the chains depending on the color at
the connection vertices or $\mathbf a$. If we had not imposed
\eqref{eq:weightcond}, we would also have chains for each set of marked
points of total weight exactly $1$.

\begin{figure}
  \caption{An edge and a point chain}
  \label{fig:chain}
  \begin{tikzpicture}[label distance=-2.5mm]
    \fill (-2.5, 0) circle (2pt) node[left] {$0$} node[right] {$+1$}
    (-2.5, 1) circle (2pt) node[left] {$\infty$} node[right] {$-1$};
    \draw (-2.5, -0) -- (-2.5, 0.5) node[left] {$\mathbb P^1$}
    node[right] {$\zeta$}-- (-2.5, 1);

    \draw (0, 0) node[left] {$\Mbar_{g(v), \mathbf w(v)}$} \foreach \i
    in {1,...,5} {-- ($\i*(0.4, 0) + (-0.2, 0.5)$) node[label=3*(-1)^\i:$d_{e_\i}$] {}
      -- ($\i*(0.4, 0) + (0, 0.5) + 0.5*pow(-1,\i+1)*(0,1)$)
      node[label=-90+180*\i:$d_{v_\i}$] {} } -- (2.2,0.5) node[right]
    {$d_{e_6}$} -- (2.4,0) node[right] {$\Mbar_{g(w), \mathbf w(w)}$};
    
    \begin{scope}[xshift=6cm]
      \draw (0, 1) node[left] {$\Mbar_{g(v), \mathbf w(v)}$} \foreach
      \i in {1,...,4} {-- ($\i*(0.4, 0) + (-0.2, 0.5)$)
        node[label=-3*(-1)^\i:$d_{e_\i}$] {} --
        ($\i*(0.4, 0) + (0, 0.5) - 0.5*pow(-1,\i+1)*(0,1)$)
        node[label=90+180*\i:$d_{v_\i}$] {} } -- (1.8,0.5) node[right]
      {$d_{e_5}$} -- (2,0) node[right] {point};
  \end{scope}
  \end{tikzpicture}
\end{figure}
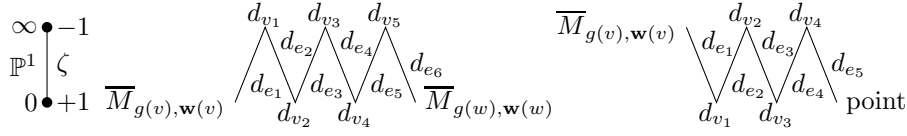

The fixed loci corresponding to such a decorated stable graph are, up
to a finite map, isomorphic to products
\begin{equation*}
  \prod_v \Mbar_{g(v), (\mathbf w(v), \varepsilon^{d(v)})}/S_{d(v)}
\end{equation*}
times a number of Losev-Manin factors $\Mbar_{0, 2 | d}/S_d$
corresponding to the vertices inside the chains.

For the localization calculation we will also have to consider the
pull-back of the integrand $\varphi(\mathbf a)$ to each fixed
locus. The factors $\ev_i^*(([0] + [\infty])^{a_i})$ of
$\varphi(\mathbf a)$ merely restrict the coloring of the stable
quotient graphs and their contribution is pulled back from the
equivariant cohomology of a point. The other factors $s_i^{a_i}$ need
to be partitioned along the factors of the stable quotient fixed
locus. However it is easy to see that in order for the contribution to
be nonzero $s_i^{a_i}$ has to land at the factor corresponding to the
vertex $p(i)$.

After all these preconsiderations let us write down the localization
formula. It will be the most convenient to write it in a power series
form. We have
\begin{equation*}
  \nu_*^{vir}\varphi(\mathbf a) = \left[\sum_{\Gamma, \zeta, d} \frac 1{\Aut(\Gamma)} \xi_{\Gamma *} \varepsilon_* \left(\prod_v
      \mathrm{Vertex}^1_v \prod_e \mathrm{Edge}^1_e
      \prod_{i = n' + 1}^n \mathrm{Point}_i^1\right)\right]_{x^d\mathbf p^{\mathbf a}},
\end{equation*}
where $\nu_*^{vir}$ again denotes push-forward using $\nu$ after
capping with the virtual fundamental class and the vertex, edge and
point series still need to be defined.

The vertex series $\mathrm{Vertex}^1_v$ is given by
\begin{equation*}
  \mathrm{Vertex}^1_v = (\zeta(v)s)^{g(v) - 1} \exp(-T_1) T_2 \in \bigoplus_{d = 0}^\infty A_{\mathbb C^*}^*(\Mbar_{g(v), (\mathbf w(v), \varepsilon^d)}) \otimes \mathbb Q[s,s^{-1}][\![x, \mathbf p]\!]
\end{equation*}
with
\begin{align*}
  T_1 =& \sum_{n' \ge i \in p^{-1}(v)} s_i p_i + \frac 12\zeta(v) \sum_{n' \ge i \in p^{-1}(v)} sp_i\\
  T_2 =& \sum_{j = 0}^\infty \sum_{d = 0}^\infty (\zeta(v) s)^{d - j}
  c_j(\mathbb F_d) \frac{x^d}{(\zeta(v) s)^{2d} d!}.
\end{align*}

For the edge series $\mathrm{Edge}^1_e$ we have
\begin{align*}
  \mathrm{Edge}^1_e =& \sum_{d = 1}^\infty x^d \sum_{\Gamma_e^d = (V_e, E_e)} \frac 1{\Aut(\Gamma_e^d)} \\
  &\frac{\zeta_1 s}{\omega_{\Gamma_e^d, 1} - \psi_1} \frac
  {\zeta_2 s}{\omega_{\Gamma_e^d, 2} - \psi_2} \prod_{f\text{ edge}} \Cont(f)
  \prod_{v\text{ vertex}} \Cont(v),
\end{align*}
where $\Gamma_e^d$ is a stable quotient localization graph of
$\Qbar_{0, 2}(\mathbb P^1, d)$ \footnote{which is here a chain similar
  to the first in Figure~\ref{fig:chain}} with color of the two
vertices determined by $\zeta$ at the two vertices adjacent to
$e$. The variables $\zeta_i$, $d_i$, $\psi_i$ for $i \in \{1, 2\}$
denote the color, the weight and the $\psi$-classes of the two marked
points corresponding to the edge respectively. The contributions
$\Cont(f)$ and $\Cont(v)$ are contributions to the calculation of the
integrals
\begin{equation*}
  \int\limits_{\Qbar_{0, 2}(\mathbb P^1, d)} \ev_1^*(\phi_1) \ev_2^*(\phi_2)
\end{equation*}
as in Section~\ref{sec:locsum}. By the stable quotients Gromov-Witten
comparison we can replace $\Qbar_{0, 2}(\mathbb P^1, d)$ everywhere in
this paragraph by $\Mbar_{0, 2}(\mathbb P^1, d)$ and then
$\mathrm{Edge}^1_e$ becomes very similar to $E^{ij}(t_1, t_2, x)$ from
Section~\ref{sec:locsum}.

Finally the point series $\mathrm{Point}_i^1$ is similarly given by
\begin{align*}
  \mathrm{Point}_i^1 = &\left(\frac s2 \zeta_i\right)^{a_i} + \\
  &\sum_{d = 1}^\infty x^d \sum_{\Gamma_i^d = (V_i, E_i)} \frac
  1{\Aut(\Gamma_i^d)} \frac{\zeta_i s}{\omega_{\Gamma_i^d} - \psi_i}
  \prod_{f\text{ edge}} \Cont(f) \prod_{v\text{ vertex}} \Cont(v),
\end{align*}
where $\zeta_i$ is $\zeta$ at the vertex with $i$, $\Gamma_i^d$ is a
stable quotient localization graph of $\Qbar_{0, 2}(\mathbb P^1, d)$
with color of the first vertex determined by $\zeta_i$. Here the
contributions $\Cont(f)$ and $\Cont(v)$ are contributions to the
calculation of the integrals
\begin{equation*}
  \int\limits_{\Qbar_{0, 2}(\mathbb P^1, d)} \ev_1^*(\phi_1) \ev_2^*\left(\phi_2 \left(\frac{[0] + [\infty]}2\right)^{a_i}\right).
\end{equation*}
The first summand corresponds to the case when the length of the chain
is $0$. Its form comes from the identity $[0] - [\infty] = s$ and the
fact that $[0] \cdot [\infty] = 0$.

\subsubsection{Pushing forward}

Next we will study the $\varepsilon$-push-forward in the formula for
$\nu_*^{vir}\varphi(\mathbf a)$. For this we want to replace the
$\psi$-classes in the edge and point terms by $\psi$-classes
pulled back via $\varepsilon$, since the other $\psi$-classes are
already pull-backs. For a fixed localization graph $\varepsilon$ is
a composition of local maps, one for each vertex in the graph, of the
form $\Mbar_{g(v), (w(v), \varepsilon^{d(v)})} \to \Mbar_{g(v),
  w(v)}$, and one for each edge and marked point which contracts
products of factors of the form $\Mbar_{0, 2 | d}$.

Let us first look at the push-forward local to a vertex. For this we
only need to look at the product of the vertex term with the factors
of the form
\begin{equation*}
  \frac 1{\omega_i - \psi_i}  
\end{equation*}
from the adjacent edge and point series. Let us simplify the notation
for this local discussion. With $d$ we denote the degree at this
vertex, with $n = n' + n''$ the number of marked points with weight
different or equal to $1$ respectively, and with $\mathbf w$ the
weights at the vertex. We will index the marked points by the set
$\{1, \dotsc, n\} \sqcup \{1, \dotsc, d\}$. Hopefully the non-empty
intersection of these sets will not cause any confusion.

The basic pull-back formula is that
\begin{equation*}
  \psi_i = \varepsilon^*(\psi_i) + \Delta_{i1}
\end{equation*}
in the case that $d = 1$. Here $\Delta_{i1}$ is the boundary divisor
of curves who have one irreducible component of genus $0$ containing
only $i$ and the weight $\varepsilon$ point. This generalizes to the
formula
\begin{equation*}
  \psi_i^k = \varepsilon^*(\psi_i^k) + \varepsilon^*(\psi_i^{k - 1}) \sum_{\emptyset \neq T \subset \{1, \dotsc, d\}} \Delta_{iT},
\end{equation*}
where $\Delta_{iT}$ is the boundary divisor of curves who have one
genus $0$ irreducible component containing only $i$ and the weight
$\varepsilon$ points indexed by $T$. Thus
\begin{equation*}
  \frac 1{\omega_i - \psi_i} = \frac 1{\omega_i - \varepsilon^*(\psi_i)} \left(1 +\omega_i^{-1} \sum_{\emptyset \neq T \subset \{1, \dotsc, d\}} \Delta_{iT}\right).
\end{equation*}

Modulo factors pulled back via $\varepsilon$ the most general classes
we will need to push forward are products of factors of the form
\begin{itemize}
\item $D_{iS}$ for $i \in \{1, \dotsc, n'\}$, $S \subset \{1, \dotsc,
  d\}$,
\item $\Delta_{iT}$ for $i \in \{n' + 1, \dotsc, n\}$, $T \subset \{1,
  \dotsc, d\}$,
\item $M$, a monomial in $\psi$- and diagonal classes of the $d$
  points of weight $\varepsilon$.
\end{itemize}
In order for the push-forward to be nonzero the sets $S$ and $T$ must
be pairwise disjoint. Moreover for each factor $\Delta_{iT}$, the
diagonal class corresponding to $T$ must be a connected factor of $M$.

Therefore for a monomial $M = \prod_{i \in P} M_i$ corresponding to a
set partition $P \vdash \{1, \dotsc, d\}$ we obtain
\begin{multline*}
  \varepsilon_*\left(\prod_{i = n' + 1}^n \left(1 + \sum_{\emptyset \neq T \subset \{1, \dotsc, d\}} \omega_i^{-1} \Delta_{iT}\right) \exp\left(\sum_{i = 1}^{n'} s_ip_i\right) M\right) = \\
    \prod_{i \in P} \left(\delta^{\Delta}_i \sum_{j = n' + 1}^n \omega_j^{-1} + \varepsilon_{i*}\left(\exp\left(\sum_{j = 1}^{n'} s_jp_j\right) M_i\right)\right),
\end{multline*}
where here $\delta^{\Delta}_i$ is one if $M_i$ is a diagonal class and
zero otherwise and the $\varepsilon_i$ are forgetful maps
$\varepsilon_i: \Mbar_{g, (\mathbf w, \varepsilon^{|i|})} \to
\Mbar_{g, \mathbf w}$
\footnote{To think of $M_i$ as living on
  $\Mbar_{g, (\mathbf w, \varepsilon^{|i|})}$ one needs to choose a
  bijection $i \to \{1, \dotsc, |i|\}$ but the
  $\varepsilon_i$-push-forward is independent of that choice.}. Notice
that the push-forward is a product where each factor is dependent only
on one factor of $M$.

This allows us to use the arguments from
Section~\ref{open:proof:simple} to calculate the necessary
$\varepsilon$-push-forwards modulo classes pulled back via
$\varepsilon$. We obtain
\begin{equation}
  \label{eq:localpushforward}
  \begin{aligned}
    &\begin{aligned}
      \sum_{d = 0}^\infty \varepsilon_* \sum_{i = n' + 1}^n \left(1 + \sum_{\emptyset \neq T \subset \{1, \dotsc, d\}} \omega_i^{-1} \Delta_{iT}\right) \exp\left(\sum_{i = 1}^{n'} s_ip_i\right) \\
      \sum_{j = 0}^\infty (\zeta s)^{d - j} c_j(-\mathbb B_d)
      \frac{x^d}{(\zeta s)^{2d} d!} =
    \end{aligned}\\
    &\begin{aligned} \prod_{i = n' + 1}^n
      &\exp\left(\frac{u^{\zeta}}{\omega_i} -
        \log(\Phi')\left(\frac{\psi_i}{\zeta(v)s}, \frac
          x{\zeta(v)s}\right)\right) \\
      &\exp\left(\left\{\exp(-pD)\log(\Phi)\left(\frac t{\zeta(v)s},
            \frac x{\zeta(v)s}\right)\right\}_\Delta\right)\Big|_{t =
        1},
    \end{aligned}
  \end{aligned}
\end{equation}
where $u^\zeta$ is as in \eqref{eq:canonicalcoord} and $p = p_1 +
\dotsb + p_{n'}$. Here we rather artificially introduced a variable
$t$ to make use of the bracket notation.  Notice that in the first
factor of this formula the factor corresponding to some $i \in \{n' +
1, \dotsc, n\}$ depends only on $\omega_i$, which dependends on the
degree splitting of the chain, whereas the second factor is
independent of the degree splittings.

\smallskip

Now we can again step back from the vertex and look at the global
picture. With \eqref{eq:localpushforward} we can give a new formula
for $\nu_*^{vir}\varphi(\mathbf a)$
\begin{equation*}
  \nu_*^{vir}\varphi(\mathbf a) = \left[\sum_{\Gamma, \zeta} \frac 1{\Aut(\Gamma)} \xi_{\Gamma *} \left(\prod_v
      \mathrm{Vertex}^2_v \prod_e \mathrm{Edge}^2_e
      \prod_{i = 1}^n \mathrm{Point}_i^2\right)\right]_{x^d\mathbf p^{\mathbf a}},
\end{equation*}
with new vertex, edge and point series. The term $\mathrm{Vertex}^2_v$
already has the form as in the stable quotient relations
\begin{align*}
  &\mathrm{Vertex}^2_v = (\zeta(v) s)^{g(v) - 1} \\
  &\exp\left(-\frac 12\zeta(v) sp_{(v)} + \left\{\exp(p_{(v)}sD) \log(\Phi)\left(\frac t{\zeta(v)s}, \frac x{(\zeta(v)s)^2}\right)\right\}_{\Delta^{(v)}}\Big|_{t = 1} + V'^{\zeta(v)}\right),
\end{align*}
where
\begin{equation*}
  V'^\zeta = \sum_{j = 0}^\infty (\zeta s)^{-j} c_j(\mathbb E^*) \in A^*(\Mbar_{g(v), \mathbf w_v})[s,s^{-1}][\![x, s, s^{-1}]\!].
\end{equation*}
Furthermore the edge and point series are given by
\begin{align*}
  \mathrm{Edge^2}_e =& {\Phi'}^{-1}\left(\frac{\psi_1}{\zeta_1s},
    \frac x{\zeta_1s}\right)
  {\Phi'}^{-1}\left(\frac{\psi_2}{\zeta_2s}, \frac
    x{\zeta_2s}\right) \\
  \sum_{d = 1}^\infty x^d \sum_{\Gamma_e^d = (V_e, E_e)}& \frac
  1{\Aut(\Gamma_e^d)} \frac{\zeta_1 s
    e^{u^{\zeta_1}/\omega_{\Gamma_e^d,1}}}{\omega_{\Gamma_e^d,1} -
    \psi_1} \frac{\zeta_2 s
    e^{u^{\zeta_2}/\omega_{\Gamma_e^d,2}}}{\omega_{\Gamma_e^d,2} -
    \psi_2} \prod_{f\text{ edge}} \Cont(f) \prod_{v\text{ vertex}}
  \Cont(v)
\end{align*}
and
\begin{align*}
  \mathrm{Point}_i^2 =& {\Phi'}^{-1}\left(\frac{\psi_i}{\zeta_is},
    \frac
    x{\zeta_is}\right) \\
  \sum_{d = 0}^\infty x^d &\sum_{\Gamma_i^d = (V_i, E_i)} \frac
  1{\Aut(\Gamma_i^d)} \frac{\zeta_i s
    e^{u^{\zeta_i}/\omega_{\Gamma_i^d}}}{\omega_{\Gamma_i^d} - \psi_i}
    \prod_{f\text{ edge}} \Cont(f) \prod_{v\text{ vertex}} \Cont(v).
\end{align*}

Using Mumford's formula (pulled back from $\Mbar_g$ to $\Mbar_{g,
  \mathbf w}$) we find
\begin{align*}
  V'^\zeta = &-\sum_{i \ge 1} \frac{B_{2i}}{2i \cdot (2i - 1)}
  \kappa_{2i - 1} (\zeta s)^{1 - 2i} \\
  &-\sum_{i \ge 1} \frac{B_{2i}}{2i \cdot (2i - 1)} \sum_{\Delta
    \in \mathcal{DG}} \frac 1{\Aut(\Delta)} \xi_{\Delta,
    *}\left(\frac{\psi_a^{2i - 1} + \psi_b^{2i - 1}}{\psi_a +
      \psi_b}\right) (\zeta s)^{1 - 2i}.
\end{align*}

Notice that the edge series $\mathrm{Edge}_e^2$ modulo a slight change
in notation and the ${\Phi'}^{-1}$ factors is the series
$E^{ij}(\psi_1, \psi_2, x)$ from Section~\ref{sec:locsum}, where $i$
and $j$ correspond to the color of the vertices $e$
connects. Similarly we identify the point term $\mathrm{Point}_i^2$ up
to the ${\Phi'}^{-1}$-factor with
\begin{equation*}
  2^{-a_i} \left(s^{a_i} P^{i0}(\psi_i, x) + (-s)^{a_i} P^{i\infty}(\psi_i, x)\right).
\end{equation*}

We finally obtain the relations by taking the $s^c$ part of
$\nu_*^{vir}\varphi(\mathbf a)$ for $c < 0$. So we replace everywhere
$s$ by $t^{-1}$, $x$ by $xt^2$ and $p_{(v)}$ by $p_{(v)}t^{-1}$. After dividing
out a common factor of $t^e$ for
\begin{equation*}
  e = \sum_v (-g(v) + 1) + 2d + |a| = -g + 1 + |E| + 2d + |a|
\end{equation*}
and introducing variables $p_i$ for the $a''_i$ we arrive at the stable
quotient relations of Section~\ref{sec:gen:statement}.

\subsection{Evaluation of the relations}

\subsubsection{Minor simplification}

With the same proof as in Section~\ref{sec:open:diffalg} the
stable quotient relations are implied from the stable quotient
relations in the case that the $n'$-tuple $\mathbf a'$ is only $\{0,
1\}$-valued. The same holds trivially for $\mathbf a''$ because of
the form of the point term.

Furthermore the relations in the case that a point is of weight $1$
and a point is of weight slightly smaller than $1$ are the same. This
is because a point $i$ of weight slightly smaller than $1$ is not
allowed to meet any other point, therefore the contribution of that
point, which is solely in the vertex contribution, is
\begin{equation*}
  \exp\left(-\frac 12\zeta(p(i)) p_i + p_i D\gamma(t\zeta(p(i))\psi_i, x)\right) \equiv 1 + p_i \zeta(p(i))\delta(t\zeta(p(i))\psi_i, x) \pmod{p_i^2}
\end{equation*}
This is the point term after suitably renaming $p_i$.

Therefore we can from now on treat points of weight $1$ the same way
as points of weight slightly less than $1$.

\subsubsection{Edge terms}

The factor $\exp(V^\zeta)$ of the vertex contribution to the stable
quotient relations contains intersections of classes supported on
divisor classes of $\Mbar_{g(v), \mathbf w_v}$. We want to reformulate
the relations such that the vertex term only contains $\kappa$-,
diagonal and $\psi$-classes corresponding to the markings. Some excess
intersection calculations will be necessary to deal with
$\exp(V^\zeta)$.

\begin{prop}
  \label{prop:gen:excess}
  The set of stable quotient relations is equivalent to the following
  set of relations: Under the conditions of
  Proposition~\ref{prop:gen:sqrel} it holds
  \begin{align*}
    0 = \Bigg[\sum_{\substack {\Gamma \in \mathcal G \\ \zeta: \Gamma
        \to \{\pm 1\}}} \frac 1{\Aut(\Gamma)} \xi_{\Gamma
      *}\Big(\prod_v \mathrm{Vertex^4}_v^{\zeta(v)} & \prod_e
    \mathrm{Edge^4}_e^{\zeta(v_1), \zeta(v_2)}\Big)\Bigg]_{t^{r -
        |E|}x^d\mathbf p^{\mathbf a}},
  \end{align*}
  with
  \begin{equation*}
    \mathrm{Vertex^4}_v^\zeta = \zeta^{g(v) - 1} \exp\left(\frac 12 \zeta p_{(v)} + \{\exp(p_{(v)}D) \gamma(t\zeta, x)\}_{\Delta^{(v)}}\right),
  \end{equation*}
  where $p_{(v)} = \sum_{i \in p^{-1}(v)} p_i$, and
  \begin{equation*}
    t(\psi_1 + \psi_2)\mathrm{Edge^4}_e^{\zeta_1, \zeta_2} =
    \frac{\zeta_1 + \zeta_2}2\exp(-\gamma'(t\zeta_1\psi_1) - \gamma'(t\zeta_2\psi_2)) + \zeta_1\delta(t\zeta_1\psi_1) + \zeta_2\delta(t\zeta_2\psi_2),
  \end{equation*}
  where $\gamma'$ is defined in the same way from $\Phi'$ as $\gamma$
  is from $\Phi$:
  \begin{equation*}
    \gamma' = \sum_{i \ge 1} \frac{B_{2i}}{2i \cdot (2i - 1)} t^{2i - 1} + \log(\Phi')
  \end{equation*}
\end{prop}

\begin{rem}
  As in Section~\ref{sec:open:diffalg} we can also write
  \begin{equation*}
    \mathrm{Vertex^4}_v^\zeta = \zeta^{g(v) - 1} \exp\left(-\{\gamma\}_{\kappa^{(v)}}^\zeta + \sum_{i = 1}^\infty \frac{\zeta^i}{i!}\{p_{(v)}^i D^{i - 1} \delta\}_{\Delta^{(v)}}^\zeta\right).
  \end{equation*}
  The power $\zeta^i$ appears because of the $t$ in $D =
  tx\frac{\mathrm d}{\mathrm dx}$.
\end{rem}

The proof of Proposition~\ref{prop:gen:excess} depends on the
following lemma.

\begin{lem}
  \label{lem:lambda}
  For a polynomial $f$ in two variables we have
  \begin{align*}
    \exp\left(\sum_{\Delta \in \mathcal{DG}} \frac 1{|\Aut(\Delta)|} \xi_{\Delta, *}(f(\psi_a, \psi_b))\right) = \\
    \sum_{\Gamma \in \mathcal G}\frac 1{|\Aut(\Gamma)|} \xi_{\Gamma,
      *}\left(\prod_e \frac{\exp(-f(\psi_1^{(e)},
        \psi_2^{(e)})(\psi_1^{(e)} + \psi_2^{(e)})) -
        1}{-(\psi_1^{(e)} + \psi_2^{(e)})}\right),
  \end{align*}
  where $\psi_i^{(e)}$ are the two cotangent line classes belonging to
  edge $e$.
\end{lem}
\begin{proof}
  Formally expanding the left hand side using the intersection
  formulas, for example described in \cite[Appendix~A]{GrPa03}, we
  can write it as a sum over stable graphs $(\Gamma, E)$. Let us look
  at the term corresponding to a given graph $\Gamma$. By contracting
  all but one edge of $\Gamma$ one obtains a divisor graph. This
  process gives a map $e_\Gamma: E \to \mathcal{DG}$. Counting the
  preimages of $e_\Gamma$ gives a map $m_\Gamma: \mathcal{DG} \to
  \mathbb N_0$. In the formal expansion of the exponential on the left
  hand side each term also corresponds to a function $\sigma:
  \mathcal{DG} \to \mathbb N_0$.

  A term contributes to a graph $\Gamma$ if and only if $m_\Gamma \le
  \sigma$, of the $|\sigma|$ intersections $|m_\Gamma| = |E|$ are
  transversal and the others are excess. In addition, a contributing
  term of the left hand side determines a partition $\mathbf p$
  indexed by $E$ of $\sigma = \sum_{e \in E} p_e$ such that
  $p_e(\Delta) = 0$ unless $e_\Gamma(e) = \Delta$.

  With this we can explicitly write down $|\Aut(\Gamma)|$ times the
  $\Gamma$-contribution as
  \begin{align*}
    &\xi_{\Gamma, *} \sum_{\sigma \ge m_\Gamma} \sum_{\mathbf p}
    &&\prod_{\Delta
      \in \mathcal{DG}} \frac 1{\sigma(\Delta)!} \binom{\sigma(\Delta)}{\mathbf p(\Delta)} \\
    &&&\prod_e f(\psi_1^{(e)}, \psi_2^{(e)})^{p_e(e_\Gamma(e))}
    (-(\psi_1^{(e)} +
    \psi_2^{(e)}))^{p_e(e_\Gamma(e)) - 1} \\
    = &\xi_{\Gamma, *} \sum_{\sigma \ge m_\Gamma} \sum_{\mathbf p}
    &&\prod_e \frac 1{(p_e(e_\Gamma(e)))!}  f(\psi_1^{(e)},
    \psi_2^{(e)})^{p_e(e_\Gamma(e))} (-(\psi_1^{(e)} +
    \psi_2^{(e)}))^{p_e(e_\Gamma(e)) - 1} \\
    = &\xi_{\Gamma, *} &&\prod_e \sum_{i = 1}^\infty \frac 1{i!}
    f(\psi_1^{(e)}, \psi_2^{(e)})^i (-(\psi_1^{(e)} +
    \psi_2^{(e)}))^{i - 1} \\
    = &\xi_{\Gamma, *} &&\prod_e \frac{\exp(-f(\psi_1^{(e)},
      \psi_2^{(e)})(\psi_1^{(e)} + \psi_2^{(e)})) - 1}{-(\psi_1^{(e)}
      + \psi_2^{(e)})}.
  \end{align*}
  Here the factor $(\sigma(\Delta)!)^{-1}$ comes from the exponential
  and
  \begin{equation*}
    \prod_{\Delta \in \mathcal{DG}} \binom{\sigma(\Delta)}{\mathbf p(\Delta)}
  \end{equation*}
  comes from the choice of which intersections are excess.

  Summing the contributions for all $\Gamma$ finishes the proof.
\end{proof}

\begin{proof}[Proof of the proposition]
  We will apply the lemma in the case that
  \begin{equation*}
    f(x_1, x_2) = -\sum_{i \ge 1} \frac{B_{2i}}{2i(2i - 1)} \frac{x_1^{2i - 1} + x_2^{2i - 1}}{x_1 + x_2},
  \end{equation*}
  but now we also need to take care of the coloring of the vertices.

  For each graph $\Gamma \in \mathcal G$ with a coloring $\zeta:
  \Gamma \to \{\pm 1\}$ we can construct a new graph $\Gamma_{red}$,
  its \emph{reduction}, by contracting all edges of $\Gamma$
  connecting two vertices of the same color. The induced coloring on
  $\Gamma_{red}$ satisfies the property that neighboring vertices are
  differently colored. Let us call such a graph reduced. Having the
  same reduction also defines an equivalence relation on $\mathcal G$.

  The idea is now to apply lemma \ref{lem:lambda} to each vertex of
  each graph $\Gamma$. In this way we get terms at each specialization
  $\Gamma'$ of $\Gamma$ in the same equivalence class of $\Gamma$.

  Let us collect all the different contributions at a graph $\Gamma'$
  coming from graphs $\Gamma$. Recall the pull-back formula for the
  $\kappa$-classes
  \begin{equation*}
    p_{v*}(\xi_{\Gamma}^* \kappa_i) = \kappa_i + \sum_e \psi_e^i,
  \end{equation*}
  where $p_v$ denotes the projection map to the factor corresponding
  to each vertex $v$ and the sum is over all outgoing edges at $v$.
  This implies that the contributions at $\Gamma'$ all have the same
  vertex contribution up a sign and a factor
  \begin{equation*}
    \exp(-\gamma'(t\zeta_1\psi_1^{(e)})-\gamma'(t\zeta_2\psi_2^{(e)}))
  \end{equation*}
  for each edge of $\Gamma$ not in $\Gamma'$ \footnote{$\gamma'$
    appears here instead of $\gamma$ because $\kappa_{-1} = 0$ while
    $\psi^{-1}$ is not defined.}. The edge terms corresponding to
  common edges do exactly coincide. The different factors split into a
  product over the connected components of the graph obtained by
  removing the edges which need to be contracted to obtain
  $\Gamma'_{red}$.

  So let us look at just one connected component $\Gamma'' \subset
  \Gamma' \setminus \Gamma'_{red}$. We have to sum over the possibilities $E
  \subseteq E(\Gamma'')$ of contracting edges in $\Gamma''$. We thus
  have the contribution
  \begin{multline*}
    \sum_{E(\Gamma'') = E \coprod F} \zeta^{|E|} \prod_{e \in F}
    \mathrm{Edge^3}_e^{\zeta, \zeta} \\
    \prod_{e \in E} \frac{\exp(-f(t\zeta\psi_1^{(e)},
      t\zeta\psi_2^{(e)})(t\zeta\psi_1^{(e)} + t\zeta\psi_2^{(e)})) -
      1}{-(t\zeta\psi_1^{(e)} +
      t\zeta\psi_2^{(e)})} \exp(-\gamma'(t\zeta\psi_1^{(e)})-\gamma'(t\zeta\psi_2^{(e)})) \\
    =\prod_{e \in E(\Gamma'')} \Big(\mathrm{Edge^3}_e^{\zeta, \zeta}
    + \\
    \zeta \frac{\exp(-f(t\zeta\psi_1^{(e)},
      t\zeta\psi_2^{(e)})(t\zeta\psi_1^{(e)} + t\zeta\psi_2^{(e)})) -
      1}{-(t\zeta\psi_1^{(e)} +
      t\zeta\psi_2^{(e)})}\exp(-\gamma'(t\zeta\psi_1^{(e)})-\gamma'(t\zeta\psi_2^{(e)}))\Big) \\
    =\prod_{e \in E(\Gamma'')} \mathrm{Edge^4}_e^{\zeta, \zeta}.
  \end{multline*}
  Because of
  \begin{equation*}
    \mathrm{Edge^3}_e^{\zeta, -\zeta} = \mathrm{Edge^4}_e^{\zeta, -\zeta}
  \end{equation*}
  we can replace $\mathrm{Edge^3}$ by $\mathrm{Edge^4}$ also for the
  edges connecting differently colored vertices.
\end{proof}

\subsubsection{Variable transformations}
\label{sec:general:trans}

Using the results of Section~\ref{sec:open:subst} we can give a new
formulation of the stable quotient relations.

We have
\begin{align*}
  0 = \Bigg[\sum_{\substack {\Gamma \in \mathcal G \\ \zeta: \Gamma
      \to \{\pm 1\}}} \hspace{-2mm}\frac 1{\Aut(\Gamma)} \xi_{\Gamma
    *}\Big((1 + 4y)^{e_\Gamma}\prod_v \mathrm{Vertex^5}_v^{\zeta(v)}
  \prod_e \mathrm{Edge^5}_e^{\zeta(v_1), \zeta(v_2)} \Big)\Bigg]_{u^{r
      - |E|}y^d\mathbf p^{\mathbf a}},
\end{align*}
with
\begin{equation*}
  \mathrm{Vertex^5}_v^\zeta = \zeta^{g(v) - 1} \exp\left(-\left\{\sum_{k = 1}^\infty \sum_{j = 0}^k c_{k, j} u^k y^j\right\}_\kappa^\zeta + \sum_{i = 1}^\infty \frac{\zeta^i}{i!}\{p_{(v)}^i\delta_i\}_\Delta^\zeta\right),
\end{equation*}
\begin{align*}
  u(\psi_1 + \psi_2)\mathrm{Edge^5}_e^{\zeta_1, \zeta_2} =
  \frac{\zeta_1 + \zeta_2}2 \exp\left(-\sum_{k = 1}^\infty \sum_{j =
      0}^k c_{k, j} (u\zeta_1\psi_1)^k + (u\zeta_2\psi_2)^k)
    y^j\right) \\
  + \zeta_1\delta_1(u\zeta_1\psi_1) + \zeta_2\delta_1(u\zeta_2\psi_2)
\end{align*}
and the exponent
\begin{equation*}
  e_\Gamma = \frac{r + 2d - 2}2 - \frac{\kappa_0}4 - \frac{|\mathbf a|}2 = \frac{r - g + 2d - 1 - |\mathbf a|}2
\end{equation*}
under the condition of Proposition~\ref{prop:gen:sqrel} on $r$.

We can assume that $e_\Gamma$ is integral because otherwise the
relation is zero since the term corresponding to a coloring $\zeta$
and the opposite coloring $-\zeta$ exactly cancel each other in this
case.

Next we look as in Section~\ref{sec:open:ext} at the extremal
coefficients of this series and obtain the FZ relations of
Proposition~\ref{prop:fz}.

\subsection{Final remarks}
\label{sec:final}

Let $\mathcal R(g, \mathbf w, r; S)$ denote the relation on $\Mbar_{g,
  \mathbf w}$ of Proposition~\ref{prop:fz} in codimension $r$
corresponding to $S \subset \{1, \dotsc, n\}$ viewed as a class in the
formal strata algebra, i.e. the formal $\mathbb Q$-algebra generated
by the symbols
\begin{equation*}
  \xi_{\Gamma*} \left(\prod_v M_v\right),
\end{equation*}
where $\Gamma$ is a stable graph of $\Mbar_{g, \mathbf w}$ and the
$M_v$ are formal monomials in $\kappa$-, $\psi$- and diagonal classes,
modulo the relations given by the formal multiplication rules for
boundary strata described in \cite[Appendix~A]{GrPa03} and the
relations between diagonal and $\psi$- classes from
\ref{sec:open:proof:universal}. One can describe formal analogs of the
push-forwards and pull-backs along the forgetful, gluing and weight
reduction maps.

By the way we have constructed the stable quotient relations, for
$\mathbf w' \le \mathbf w$ the push-forward of $\mathcal R(g, \mathbf
w, r; S)$ via the weight reduction map is $\mathcal R(g, \mathbf w',
r; S)$. Therefore the relations of Proposition~\ref{prop:fz} are (up
to a constant factor) the push-forward of a subset of Pixton's
generalized FZ relations.

As mentioned in the introduction more relations than in
Proposition~\ref{prop:fz} can be obtained by taking for a partition
$\sigma$ with no part equal to $2\pmod{3}$ the class
\begin{equation*}
  \mathcal R(g, (\mathbf w, 1^{\ell(\sigma)}), r -
  |\lfloor\sigma/3\rfloor|; \bar S)\prod_i \psi_{n + i}^{\lfloor\sigma_i/3\rfloor + 1}  
\end{equation*}
in $A^{r + \ell(\sigma)}(\Mbar_{g, (\mathbf w, 1^{\ell(\sigma)})})$,
where $\bar S$ equals $S$ on the first $n$ markings and is given by
the remainders when dividing the parts of $\sigma$ by $3$ on the other
markings, and pushing this class forward to $\Mbar_{g, \mathbf w}$
under the forgetful map. For explicitly calculating this push-forward
it is better to use the usual $\kappa$-classes
$\tilde \kappa_i = \pi_*(c_1(\omega_\pi(D))^{i + 1})$, which are
related to the $\kappa$-classes we have used in this article by
$\tilde \kappa_i = \kappa_i + \sum_{j = 1}^n \psi_j^i$, in order to
use Faber's formula for the push-forward of monomials in cotangent
line classes \cite{ArCo96}. Let us call these relations
$\mathcal R(g, \mathbf w, r; \sigma, S)$.

As in \cite{Pi12P} even more generally one can look at the $\mathbb
Q$-vector space $\mathcal R_{g, \mathbf w}$ generated by the relations
obtained by choosing a boundary stratum corresponding to a dual graph
$\Gamma$, taking a FZ relation $\mathcal R(g_i, \mathbf w_i, r;
\sigma, S) M_i$ for any $r$, $S$, $\sigma$ and monomial $M_i$ in the
diagonal and cotangent line classes on one of the components,
arbitrary tautological classes on the other components and pushing
this forward along $\xi_\Gamma$. Because of the compatibility with the
birational weight reduction maps \cite[Proposition~1]{Pi12P} implies
that the system $\mathcal R_{g, \mathbf w}$ of $\mathbb Q$-vector
spaces cannot be tautologically enlarged, i.e. it is closed under
formal push-forward and pull-back along forgetful and gluing maps as
well as multiplication with arbitrary tautological classes.

As in \cite{PaPi13P} we have thrown away many of the stable quotient
relations: We looked only at the extremal relations in
Sections~\ref{sec:open:ext} and \ref{sec:general:trans}. However we
should expect that these additional relations can also be expressed in
terms of FZ relations.

\printbibliography
\addcontentsline{toc}{section}{References}

\vspace{+8 pt}
\noindent
Departement Mathematik \\
ETH Zürich \\
felix.janda@math.ethz.ch

\end{document}